\documentclass{siamart190516}
\usepackage{amsfonts,amsmath,mathrsfs}
\usepackage{amssymb}
\usepackage{hyperref}
\usepackage{graphicx}
\usepackage{wrapfig}
\usepackage{subfig}
\usepackage{float}
\usepackage{bbm}
\usepackage{gensymb}
\usepackage{mathtools} % bsmallmatrix
\usepackage{siunitx}
\usepackage[a4paper]{geometry}
\usepackage[textsize=tiny,textwidth=2.3cm]{todonotes}
\newcommand{\N}{\mathbb{N}}
\newcommand{\R}{\mathbb{R}}
\newcommand{\C}{\mathbb{C}}

\newtheorem{remark}[theorem]{Remark}
\newtheorem{assumption}[theorem]{Assumption}

\newcommand{\overbar}[1]{\mkern 1.0mu\overline{\mkern-1.0mu#1\mkern-1.0mu}\mkern 1.0mu}

\title{Series reversion for electrical impedance tomography with modeling errors}

\author{H. Garde\footnotemark[1], N. Hyv\"onen\footnotemark[2], and T. Kuutela\footnotemark[2]
}

\begin{document}
\maketitle

\renewcommand{\thefootnote}{\fnsymbol{footnote}}

\footnotetext[1]{Aarhus University, Department of Mathematics, Ny Munkegade 118, 8000 Aarhus C, Denmark (garde@math.au.dk). HG is also affiliated with the Aarhus University DIGIT centre.
}

\footnotetext[2]{Aalto University, Department of Mathematics and Systems Analysis, P.O. Box 11100, FI-00076 Aalto, Finland (nuutti.hyvonen@aalto.fi, topi.kuutela@aalto.fi). The work of NH and TK was supported by the Academy of Finland (decision 336789) and Jane and Aatos Erkko Foundation via the project Electrical impedance tomography --- a novel method for improved diagnostics of stroke.}

\begin{abstract}
  This work extends the results of [Garde and Hyv\"onen,~Math.~Comp.~91:1925--1953] on series reversion for Calder\'on's problem to the case of realistic electrode measurements, with both the internal admittivity of the investigated body and the contact admittivity at the electrode-object interfaces treated as unknowns. The forward operator, sending the internal and contact admittivities to the linear electrode current-to-potential map, is first proven to be analytic. A reversion of the corresponding Taylor series yields a family of numerical methods of different orders for solving the inverse problem of electrical impedance tomography, with the possibility to employ different parametrizations for the unknown internal and boundary admittivities. The functionality and convergence of the methods is established only if the employed finite-dimensional parametrization of the unknowns allows the Fr\'echet derivative of the forward map to be injective, but we also heuristically extend the methods to more general settings by resorting to regularization motivated by Bayesian inversion. The performance of this regularized approach is tested via three-dimensional numerical examples based on simulated data.
%% NH:
  The effect of modeling errors related to electrode shapes and contact admittances is a focal point of the numerical studies.
\end{abstract}

\renewcommand{\thefootnote}{\arabic{footnote}}



\pagestyle{myheadings}
\thispagestyle{plain}
\markboth{H. GARDE, N. HYV\"ONEN, AND T. KUUTELA}{SERIES REVERSION FOR EIT}

\section{Introduction}
\label{sec:introduction}
        {\em Electrical impedance tomography} (EIT) is an imaging modality for deducing information on the admittivity distribution inside a physical body based on current and potential measurements on the object boundary; for general information on the practice and theory of EIT, we refer to the review articles \cite{Borcea02,Cheney99,Uhlmann09} and the references therein. Real-world EIT measurements are performed with a finite number of contact electrodes, and in addition to the internal admittivity, there are typically also other unknowns in the measurement setup, such as the contacts at the electrode-object interfaces and the precise positions of the electrodes. If one models the measurements with the so-called {\em smoothened complete electrode model} (SCEM)~\cite{Hyvonen17b}, a variant of the well-established standard {\em complete electrode model} (CEM) \cite{Cheng89,Somersalo92}, it is possible to include both the strength and the positions of the electrode contacts as unknowns in the inverse problem of EIT \cite{darde2021electrodeless}.

        The aim of this work is to combine the SCEM with the series reversion ideas of \cite{garde2021series} to introduce a family of numerical one-step methods with increasing theoretical accuracy for simultaneously reconstructing the internal and electrode admittivities. To be more precise, \cite{garde2021series} applied reversion to the Taylor series of a forward map, sending a perturbation in a known interior admittivity to a projected version of the current-to-voltage boundary map, in order to introduce methods of different asymptotic accuracies for reconstructing (only) the interior admittivity perturbation. Both the idealized {\em continuum model} (CM) of EIT and the SCEM were considered, but the contact admittivities were assumed to be known for the latter,~i.e.,~they were not treated as additional unknowns in the inversion. Here, we generalize the ideas of \cite{garde2021series} to include the contact admittivities as variables in the Taylor expansion and as unknowns in the subsequent series reversion. See~\cite{Arridge12} for closely related ideas.

        Unlike in~\cite{garde2021series}, we perform our analysis for general parametrizations of the internal and contact admittivities. As an example of such a parametrization, one may write the domain conductivity as $\sigma = {\rm e}^\kappa$ and treat the log-conductivity $\kappa$ as the unknown in the series reversion, which is the choice in our numerical studies. This leads to more complicated inversion formulas as the linear dependence of the involved sesquilinear forms on the unknowns is lost, but it may also have a positive effect on the reconstruction accuracy in some settings \cite{Hyvonen18}. As in \cite{garde2021series}, a fundamental requirement for the theoretical convergence of the introduced family of methods is that the Fr\'echet derivative of the forward map is injective for the chosen  parametrizations of the internal and contact admittivities; for the considered measurements with a finite number of electrodes, the injectivity is actually a sufficient condition for the convergence since it guarantees invertibility of the Fr\'echet derivative on its range that is necessarily finite-dimensional (cf.~\cite{garde2021series}). In fact, the only potentially unstable step in any of the introduced numerical methods is  the requirement to operate with the inverse of the Fr\'echet derivative. Furthermore, the asymptotic computational complexity of any of the numerical schemes is the same as that of a straightforward linearization~\cite{garde2021series}.

        Our numerical examples concentrate on testing the series reversion approach in a realistic three-dimensional setting where the Fr\'echet derivative of the forward map is not forced to be (stably) invertible via employing sparse enough discretizations for the admittivities (cf.~\cite{Alberti2019,Alberti2020,Harrach_2019}). To perform the required inversion of the Fr\'echet derivative, we resort to certain Tikhonov-type regularization motivated by Bayesian inversion; since the assumption guaranteeing the functionality of the introduced family of methods is not met, there is no reason to expect that the theoretical convergence rates carry over to this regularized framework as such. According to our tests, the introduced second and third order methods demonstrate potential to outperform a single regularized linearization. However, at least with the chosen regularization that is not specifically designed for our recursively defined family of numerical methods, a Levenberg--Marquardt algorithm based on a couple of sequential linearizations leads on average to more accurate reconstructions than the novel methods. Recall, however, that the computational cost of several sequential linearizations is asymptotically higher than that of a single series reversion of any order, and according to our tests, the performance of sequential series reversions is approximately as good as that of a same number of linearizations. In addition to these comparisons, we computationally demonstrate that all considered numerical schemes are capable of coping to a certain extent with mismodeling of the electrode contacts if their strengths and locations are included as unknowns in the inversion in the spirit of the two-dimensional numerical tests of~\cite{darde2021electrodeless}.
%%: NH:
        Other previously introduced methods for handling unknown strengths, shapes or positions of electrode contacts in EIT include,~e.g.,~\cite{Boverman17,Boverman09,Darde12,Demidenko11b,Demidenko11a,Hyvonen17,Nissinen09,Nissinen11,Soleimani06,Vilhunen02}  

        This text is organized as follows. Section~\ref{sec:Taylor} recalls the SCEM and introduces the Taylor series for its (complete) forward operator, paying particular attention to the complications caused by the SCEM allowing the contact admittivity to vanish on parts of the electrodes. The series reversion for general parametrizations of the internal and boundary admittivities is presented in Section~\ref{sec:seriesreversion}. Sections~\ref{sec:implementation} and \ref{sec:numerics} consider the implementation details and describe our three-dimensional numerical experiments based on simulated data, respectively. Finally, the concluding remarks are listed in Section~\ref{sec:conclusions}.

        \section{SCEM and the Taylor series for its forward map}
\label{sec:Taylor}

This section first recalls the SCEM and then introduces a Taylor series representation for the associated forward operator. For more information on the SCEM and how it can be employed in accounting for uncertainty in the electrode positions when solving the inverse problem of EIT, we refer to \cite{Hyvonen17b} and \cite{darde2021electrodeless}, respectively.

\subsection{Forward model and its unique solvability}
The examined physical body is modeled by a bounded Lipschitz domain $\Omega \subset \R^n$, $n=2$ or $3$, and its electric properties are characterized by
an isotropic admittivity $\sigma: \Omega \to \C$ that belongs to
\begin{equation}
  \label{eq:sigma}
  \mathcal{S} :=  \big\{ \sigma \in L^\infty(\Omega) \  \big| \
          {\rm ess} \inf ( {\rm Re}(\sigma))  > 0 \big\}.
\end{equation}
The measurements are performed by  $M \in \N \setminus \{ 1 \}$ electrodes $\{ E_m\}_{m=1}^M$ that are identified with the nonempty connected relatively open subsets of $\partial \Omega$ that they cover. We assume $E_m \cap E_l = \emptyset$ for all $m \not= l$ and denote $E:= \cup E_m$. To begin with, the electrode contacts are modeled by a surface admittivity $\zeta: \partial \Omega \to \C$ that is assumed to satisfy
\begin{equation}
  \label{eq:zeta}
\mathcal{Z} := \big\{ \zeta \in L^\infty(E) \  \big| \  {\rm Re}( \zeta) \geq 0 \   {\rm and} \  {\rm Re}( \zeta |_{E_m}) \not \equiv 0 \  {\rm for} \ {\rm all} \ m= 1, \dots, M \big\},
\end{equation}
where the conditions on $\zeta$ hold almost everywhere on $\partial \Omega$ or $E_m$. The sets $\mathcal{Z}$ and $L^\infty(E)$ are interpreted as subsets of $L^\infty(\partial \Omega)$ via zero continuation.

A single EIT measurement corresponds to driving the net currents $I_m\in\C $, $m=1, \dots, M$, through the corresponding electrodes and measuring the resulting constant electrode potentials $U_m \in \C$, $m=1, \dots, M$. Obviously, any reasonable current pattern $I = [I_1,\dots,I_M]^{\rm T}$ belongs to the mean-free subspace
\[
\C^M_\diamond \, := \, \Big\{J \in\C^M\,\Big|\, \sum_{m=1}^M J_m = 0\Big\}.
\]
In the following analysis, the potential vector $U = [U_1,\dots,U_M]^{\rm T} \in \C^M$ is often identified with the piecewise constant function
\begin{equation}
\label{eq:piecewise}
U \, = \, \sum_{m=1}^M U_m \chi_m,
\end{equation}
where $\chi_m$ is the characteristic function of the $m$th electrode $E_m$.

According to the SCEM~\cite{Hyvonen17b}, the electromagnetic potential $u$ inside $\Omega$ and the electrode potentials $U \in L^\infty(E) \subset L^\infty(\partial \Omega)$ weakly satisfy
\begin{equation}
\label{eq:cemeqs}
\begin{array}{ll}
\displaystyle{\nabla \cdot(\sigma\nabla u) = 0 \qquad}  &{\rm in}\;\; \Omega, \\[6pt]
{\displaystyle {\nu\cdot\sigma\nabla u} = \zeta (U - u) } \qquad &{\rm on}\;\; \partial \Omega, \\[2pt]
{\displaystyle \int_{E_m}\nu\cdot\sigma\nabla u\,{\rm d}S} = I_m, \qquad & m=1,\ldots,M, \\[4pt]
\end{array}
\end{equation}
where $\nu$ is the exterior unit normal of $\partial\Omega$. Set ${\bf 1} = [1 \dots 1]^{\rm T}\in \C^M$. The variational formulation of \eqref{eq:cemeqs} is to find the unique $(u,U)$ in the space of equivalence classes
\begin{align*}
  \mathcal{H} &:= \big\{ \{ (v + c, V + c {\bf 1}) \, | \, c \in \C \} \, \big| \, (v, V) \in H^1(\Omega)\oplus \C^M \big\}
\end{align*}
such that
\begin{equation}
\label{eq:weak}
B_{\sigma,\zeta}\big((u,U),(v,V)\big)  \,=  \, I\cdot \overbar{V} \qquad {\rm for} \ {\rm all} \ (v,V) \in  \mathcal{H},
\end{equation}
with the sesquilinear form $B_{\sigma,\zeta}: \mathcal{H} \times \mathcal{H} \to \C$ defined by
\begin{equation}
\label{eq:sesqui}
B_{\sigma,\zeta}\big((w,W),(v,V)\big) = \int_\Omega \sigma\nabla w\cdot \nabla \overbar{v} \,{\rm d}x + \int_{\partial \Omega} \zeta (W-w)(\overbar{V}-\overbar{v})\,{\rm d}S.
\end{equation}
The space $\mathcal{H}$ is equipped with the standard quotient norm.
\begin{equation}
\label{eq:norm}
\|(v,V)\|_{\mathcal{H}} := \inf_{c\in\C}\Big( \|v-c\|_{H^1(\Omega)}^2 + | V - c {\bf 1}|^2 \Big)^{1/2},
\end{equation}
where $| \cdot  |$ denotes the Euclidean norm.

According to the material in \cite[Section~4]{garde2021series}, for any $(\sigma, \zeta) \in L^\infty(\Omega) \times L^\infty(E)$ it holds
\begin{equation}
 \label{eq:cont}
 \big| B_{\sigma,\zeta}\big((w,W), (v,V) \big) \big| \, \leq \, C (\|\sigma\|_{L^\infty(\Omega)} +  \|\zeta \|_{L^\infty(\partial \Omega)}) \|(w,W) \|_{\mathcal{H}}  \|(v,V) \|_{\mathcal{H}},
\end{equation}
where the constant $C > 0$ is independent of $(w,W), (v,V) \in \mathcal{H}$. This settles the continuity of the considered sesquilinear form. To tackle its coercivity in a manner that allows perturbing each contact admittivity in $\mathcal{Z}$ to any direction, we define $\mathcal{Z}'$ to be the largest subset of  $L^\infty(E)$ for which the following condition holds for any $(\sigma,\zeta) \in \mathcal{S} \times \mathcal{Z}'$:
\begin{equation}
  \label{eq:coer}
{\rm Re} \Big( B_{\sigma,\zeta} \big((v,V), (v,V) \big) \Big) \, \geq \, c \|(v,V) \|_{\mathcal{H}}^2,
\end{equation}
with $c=c(\sigma,\zeta)>0$ independent of $(v,V) \in \mathcal{H}$. According to \cite[Lemma~2.4]{darde2021electrodeless}, the set $\mathcal{Z}'$ is open in $L^\infty(E)$ and contains $\mathcal{Z}$. Moreover, it follows directly from \cite[Lemma~2.4]{darde2021electrodeless} that any $(\sigma,\zeta) \in \mathcal{S} \times \mathcal{Z}'$ has a nonempty open neighborhood $\mathcal{N}_{\sigma,\zeta}$ in $L^\infty(\Omega)\times L^\infty(E)$ such that \eqref{eq:coer} holds for all $(\varsigma, \xi) \in \mathcal{N}_{\sigma,\zeta}$ with the same constant $c > 0$.

If $(\sigma, \zeta) \in \mathcal{S} \times \mathcal{Z}'$, then \eqref{eq:weak} has a unique solution that satisfies
\begin{equation}
  \label{eq:existence}
\| (u,U) \|_{\mathcal{H}} \leq \frac{|I|}{c(\sigma,\zeta)},
\end{equation}
as can easily be deduced by resorting to \eqref{eq:cont}, \eqref{eq:coer} and the Lax--Milgram lemma. The constant $c(\sigma,\zeta)$ in \eqref{eq:existence} is the one appearing in \eqref{eq:coer}. In particular, \eqref{eq:existence} allows us to define a nonlinear ``parameter to forward solution operator'' map $N: \mathcal{S} \times \mathcal{Z}' \to \mathscr{L}(\C_\diamond^M, \mathcal{H})$ via
$$
N(\sigma, \zeta)I =  (u,U),
$$
where $(u,U)$ is the solution to \eqref{eq:weak} for $(\sigma, \zeta) \in \mathcal{S} \times \mathcal{Z}'$. Moreover, denoting by $T \in \mathscr{L}(\mathcal{H}, \C_\diamond^M)$ the `trace map' that picks the zero-mean representative for the second component of an element in $\mathcal{H}$, we define the forward map of the SCEM,~i.e.,~$\Lambda: \mathcal{S} \times \mathcal{Z}' \to \mathscr{L}(\C_\diamond^M)$, through
$$
\Lambda(\sigma, \zeta)I = T N (\sigma, \zeta)I
$$
for $(\sigma, \zeta) \in \mathcal{S} \times \mathcal{Z}'$ and $I \in \C_\diamond^M$.

\subsection{Forward map and its Taylor series}
In this section, we extend the Taylor series presented in~\cite{garde2021series} to also include the contact admittivity as a variable. This extension is based essentially on the same arguments as the ones utilized in~\cite{garde2021series}. In addition to the standard parametrization for the internal and boundary admittivities,~i.e.,~treating $\sigma$ and $\zeta$ directly as the variables, we also consider more general (analytic) parametrizations. This is motivated by the fact that real world domain conductivities range from the order of \SI{e-20}{\siemens\per\meter} (plastics) to \SI{e8}{\siemens\per\meter} (metals), which means that it may be numerically advantageous to adopt,~e.g.,~the log-admittivity $\kappa = \log \sigma$ as the to-be-reconstructed variable. However, on the downside, such an approach leads to somewhat more complicated structures for the associated Taylor series and reversion formulas due to the loss of linear dependence of the sesquilinear form \eqref{eq:sesqui} on the parameters of interest.

\subsubsection{Standard parametrization}
A general interior-boundary admittivity pair is denoted by $\tau := (\sigma, \zeta) \in \mathcal{S} \times \mathcal{Z'} =: \mathcal{K}_+$. Analogously, the space of total admittivity perturbations is denoted by $\mathcal{K} := L^{\infty}(\Omega) \times L^{\infty}(E)$ and equipped with a natural norm,~i.e.,~$\|(\sigma,\zeta)\|_\mathcal{K} := \|\sigma\|_{L^\infty(\Omega)} + \|\zeta\|_{L^\infty(E)}$. We mildly abuse the notation by writing $B_\tau$ instead of $B_{\sigma, \zeta}$ for the sesquilinear form introduced in \eqref{eq:sesqui}.

	For a fixed $\tau \in \mathcal{K_+}$, we define a linear operator $P_\tau \in \mathscr{L}(\mathcal{K}, \mathscr{L}(\mathcal{H}))$ via $P_\tau(\eta)(y, Y) = (w, W) $, where $(w, W) \in \mathcal{H}$ is the unique solution of
	\begin{equation}
		\label{eq:Pdef}
		B_{\tau} \big((w, W), (v, V)\big) = - B_{\eta}\big((y, Y), (v, V)\big) \qquad \text{for all } (v, V) \in \mathcal{H}.
	\end{equation}
        The unique solvability of \eqref{eq:Pdef} can be deduced by combining \eqref{eq:cont} and \eqref{eq:coer} with the Lax--Milgram lemma, which also directly provides the estimates
        \begin{equation}
		\label{eq:Pbound}
		\| P_\tau(\eta) \|_{\mathscr{L}(\mathcal{H})} \leq \frac{C \| \eta \|_\mathcal{K}}{c(\tau)}, \qquad \| P_\tau \|_{\mathscr{L}(\mathcal{K},\mathscr{L}(\mathcal{H}))} \leq \frac{C}{c(\tau)},
	\end{equation}
        where $C$ and $c(\tau)$ are the constants appearing in \eqref{eq:cont} and \eqref{eq:coer}, respectively. It is well known that $P_\tau(\eta)(u,U) = P_\tau(\eta) N(\tau)I \in \mathcal{H}$, with $(u,U)$ being the solution to \eqref{eq:weak} for $I \in \C_\diamond^M$ and $\tau = (\sigma, \zeta)$, is the Fr\'echet derivative of $N(\tau)I = (u,U)$ with respect to $\tau$ in the direction $\eta$; see,~e.g.,~\cite{darde2021electrodeless,Kaipio00}.

	In order to deduce higher order derivatives for $N(\tau)$, we first write down a differentiability lemma for $P_\tau$, playing the same role as \cite[Lemmas 3.1 \& 5.1]{garde2021series} in the case where only the domain admittivity is considered as a variable.

	\begin{lemma}
		\label{lem:DP}
		The map $P_\tau \in \mathscr{L}(\mathcal{K}, \mathscr{L}(\mathcal{H}))$ is infinitely times continuously Fr\'echet differentiable with respect to $\tau \in \mathcal{K}_+$. Its first derivative at $\tau$ in the direction $\eta$, i.e.~$D_\tau P_\tau(\, \cdot \, ; \eta) \in \mathscr{L}(\mathcal{K}, \mathscr{L}(\mathcal{H}))$, allows the representation
		\begin{equation}
                  \label{eq:product}
		  D_{\tau} P_\tau (\, \cdot \, ; \eta) = P_\tau(\eta) P_\tau(\, \cdot \,).
		\end{equation}
	\end{lemma}
	\begin{proof}
          The result follows through exactly the same line of reasoning as \cite[Lemmas 3.1 \& 5.1]{garde2021series}, bearing in mind that for any $\tau \in \mathcal{K}_+$ there exists a nonempty open neighborhood $\mathcal{N}_{\tau} \in \mathcal{K}_+$ such that the coercivity of the sesquilinear form $B_\tau$, characterized by \eqref{eq:coer}, holds with the same constant everywhere in $\mathcal{N}_{\tau}$; see \cite[Lemma~2.4]{darde2021electrodeless}.
        \end{proof}

         Observe that \eqref{eq:product} immediately provides means to deduce formulas for higher order derivatives for $N$ and $P_{\tau}$ via the product rule. However, as in \cite{garde2021series,Garde2020} for the mere domain admittivity, we can actually do better. To this end, let $p_k$ be the collection of index permutations of length $k$, that is,
	\begin{equation*}
		p_k = \{\alpha_1, \ldots, \alpha_k \; | \; \alpha_i \in \{1, \ldots, k\}, \; \alpha_i \neq \alpha_j \; \text{if} \; i \neq j\}.
	\end{equation*}

	\begin{theorem}
		\label{thm:lineartaylor}
		The mapping $N$ is infinitely times continuously Fr\'echet differentiable. Its derivatives at $\tau \in \mathcal{K_+}$ are given by
		\begin{equation}
			\label{eq:DkN}
			D^k N(\tau; \eta_1, \ldots, \eta_k) = \left( \sum_{\alpha \in p_k} P_\tau(\eta_{\alpha_1}) \cdots P_\tau(\eta_{\alpha_k}) \right) N(\tau), \qquad k \in \N,
		\end{equation}
		with $\eta_1, \ldots, \eta_k \in \mathcal{K}$. In particular, $N$ is analytic with the Taylor series
		\begin{equation}
			\label{eq:Ntaylor}
			N(\tau + \eta) = \sum_{k = 0}^\infty \frac{1}{k!} D^k N(\tau; \eta, \ldots, \eta) = \sum_{k = 0}^\infty P_\tau(\eta)^k N(\tau),
		\end{equation}
		where $\tau \in \mathcal{K}_+$ and $\eta \in \mathcal{K}$ is small enough so that also $\tau + \eta \in \mathcal{K}_+$.
	\end{theorem}
	\begin{proof}
		The proof follows via exactly the same arguments as \cite[Theorem 3.3 \& Theorem 5.2]{garde2021series}, with $B_\eta$ taking the role of the employed continuous (and coercive for $\eta = \tau$) sesquilinear form as per the definition of our operator $P_\tau$.
	\end{proof}

        \begin{remark}
		\label{rem:Dtrace}
		By the linearity and boundedness of the trace operator $T$, the results of Theorem~\ref{thm:lineartaylor} immediately extend for the forward map $\Lambda = T N$ of the SCEM. In particular,
		\begin{equation}
                  \label{eq:Ltaylor}
			\Lambda(\tau + \eta) = \sum_{k = 0}^\infty \frac{1}{k!} D^k \Lambda(\tau; \eta, \ldots, \eta) = T\sum_{k = 0}^\infty P_\tau(\eta)^k N(\tau)
		\end{equation}
                for $\tau \in \mathcal{K}_+$ and small enough $\eta \in \mathcal{K}$.
	\end{remark}

        \begin{remark}
          The abstract requirement that $\eta \in \mathcal{K}$ needs to have ``small enough'' norm for \eqref{eq:Ntaylor} and \eqref{eq:Ltaylor} to hold could be replaced by a more concrete condition if the real part of the second component in $\tau$,~i.e.~the contact admittivity, were required to be bounded away from zero almost everywhere on $E$. Namely, it would be sufficient to require that $\eta$ is so small that the real parts of both components of $\tau + \eta$ remain bounded away from zero on $\Omega$ and $E$, respectively. However, if the real part of the contact admittivity is allowed vanish (or just to tend to zero) on $E$, one must allow the abstract condition on the size of (the second component of) an admissible perturbation~$\eta$.
          \end{remark}

Using~\eqref{eq:Pbound}, we get via the same line of reasoning as in \cite[Corollary 5.3]{garde2021series} the following bounds for the above introduced derivatives,
	\begin{align}
		\label{eq:DnNbound}
		\big\| D^k N(\tau) \big\|_{\mathscr{L}^k(\mathcal{K}, \mathscr{L}(\C^M_\diamond, \mathcal{H}))} &\leq \frac{k! C_{\mathcal{H}}}{c(\tau)^{k+1}}, \\[1mm]
		\label{eq:DnLambdabound}
		\| D^k \Lambda (\tau)\|_{\mathscr{L}^k(\mathcal{K}, \mathscr{L}(\C^M_\diamond))} &\leq \frac{k! C_{\mathcal{H}}^2}{c(\tau)^{k+1}},
	\end{align}
where $c(\tau)$ is the $\tau$-dependent coercivity constant from \eqref{eq:coer} and $C_{\mathcal{H}} > 0$ depends on the geometric setup.

        \subsubsection{General parametrization}

        Let us then consider a general parametrization of the total admittivity, namely a mapping
        \begin{equation}
          \label{eq:cond_par}
        \tau:
        \left\{
        \begin{array}{l}
          \iota \mapsto \tau(\iota), \\[1mm]
          \mathcal{I} \to \mathcal{K}_+,
        \end{array}
        \right.
        \end{equation}
        where $\mathcal{I}$ is a nonempty open subset of a Banach space $\mathcal{B}$. We define the parametrized solution and forward maps via
        \begin{equation}
          \label{eq:parametrized_maps}
        \widetilde{N}:
        \left\{
        \begin{array}{l}
          \iota \mapsto N(\tau(\iota)), \\[1mm]
          \mathcal{I} \to \mathscr{L}(\C^M_\diamond, \mathcal{H}),
        \end{array}
        \right.
        \qquad
        \widetilde{\Lambda}:
        \left\{
        \begin{array}{l}
          \iota \mapsto \Lambda(\tau(\iota)), \\[1mm]
          \mathcal{I} \to \mathscr{L}(\C^M_\diamond),
        \end{array}
        \right.
        \end{equation}
respectively. Based on the material in the previous section, it is obvious that the regularity of the composite maps $\iota \mapsto \widetilde{N}(\iota)$ and $\iota \mapsto \widetilde{\Lambda}(\iota)$ is dictated by the properties of the parametrization $\iota \mapsto \tau(\iota)$.

\begin{corollary}
  \label{corollary:para}
          If the parametrization in \eqref{eq:cond_par} is $l$ times Fr\'echet differentiable, then so are the parametrized maps $\widetilde{N}$ and $\widetilde{\Lambda}$ defined in \eqref{eq:parametrized_maps}. Moreover, if \eqref{eq:cond_par} is analytic, then
\begin{align*}
\widetilde{N}(\iota + \eta) &= \sum_{k = 0}^\infty \frac{1}{k!} D^k \widetilde{N}(\iota; \eta, \ldots, \eta), \\
\widetilde{\Lambda}(\iota + \eta) &= \sum_{k = 0}^\infty \frac{1}{k!} D^k \widetilde{\Lambda}(\iota; \eta, \ldots, \eta),
    \end{align*}
for $\iota \in \mathcal{I}$ and small enough $\eta \in \mathcal{B}$.
            \end{corollary}
            \begin{proof}
              The result follows from the chain rule and the basic properties of analytic maps on Banach spaces; see,~e.g.,~\cite{whittlesey1965}.
              \end{proof}

            Since we know the derivatives of $\tau \mapsto N(\tau)$ and $\tau \mapsto \Lambda(\tau)$, it would be possible to introduce explicit formulas for the $k$th derivatives of the parametrized maps $\iota \mapsto \widetilde{N}(\iota)$ and $\iota \mapsto \widetilde{\Lambda}(\iota)$ based on the first $k$ derivatives of $\iota \mapsto \tau(\iota)$ and thereby to more explicitly characterize the Taylor series in Corollary~\ref{corollary:para}; see,~e.g.,~\cite[Formula~A]{Fraenkel78}. However, as we do not need such general formulas for introducing our numerical schemes, but truncated versions of the aforementioned Taylor series are sufficient for our purposes, we settle with writing down explicit expressions for the first three Fr\'echet derivatives of $\iota \mapsto \widetilde{N}(\iota)$ and $\iota \mapsto \widetilde{\Lambda}(\iota)$ in what follows.

	To retain readability of the series reversion formulas presented in Section~\ref{sec:seriesreversion}, we use the following shorthand notation for derivatives of $\iota \mapsto \tau(\iota)$ in the argument of $P_{\tau(\iota)}$,
	\begin{equation*}
		P_{\tau(\iota)}\big(D^k\tau(\iota; \eta_1, \ldots, \eta_k)\big) = P_{\tau^{(k)}}(\iota; \eta_1, \dots , \eta_k), \qquad \eta_1, \dots, \eta_k \in \mathcal{B},
	\end{equation*}
	with an even more compact variant when $\eta_1 = \cdots = \eta_m$,
	\begin{equation*}
		P_{\tau(\iota)}\big(D^k\tau(\iota; \eta, \ldots, \eta)\big) = P_{\tau^{(k)}}(\iota; \eta^k).
	\end{equation*}
        The nonlinear dependence on $\iota \in \mathcal{I}$ is often not explicitly marked for the sake of brevity in what follows.

	\begin{lemma}
          \label{lemma:tilde_deriv}
		The first three derivatives of $\iota \mapsto \widetilde{N}(\iota)$ allow the representations
		\begin{align}
                \label{eq:DNt1}  D \widetilde{N}(\iota; \eta_1) &= P_{\tau'}(\eta_1) \widetilde{N}(\iota), \\[1mm]
			D^2 \widetilde{N}(\iota; \eta_1, \eta_2) &= \big{(}P_{\tau'}(\eta_1) P_{\tau'}(\eta_2) + P_{\tau'}(\eta_2) P_{\tau'}(\eta_1) + P_{\tau''}(\eta_1, \eta_2)\big{)} \widetilde{N}(\iota), \nonumber \\[1mm]
			D^3 \widetilde{N}(\iota; \eta_1, \eta_2, \eta_3)&  = \big{(} P_{\tau'}(\eta_1 ) P_{\tau'}(\eta_2 ) P_{\tau'}(\eta_3) + P_{\tau'}(\eta_1) P_{\tau'}(\eta_3) P_{\tau'}(\eta_2) + P_{\tau'}(\eta_2 ) P_{\tau'}(\eta_1) P_{\tau'}(\eta_3) \nonumber \\
			&\qquad +P_{\tau'}(\eta_2) P_{\tau'}(\eta_3) P_{\tau'}(\eta_1) +
			P_{\tau'}(\eta_3) P_{\tau'}(\eta_1) P_{\tau'}(\eta_2) + P_{\tau'}(\eta_3) P_{\tau'}(\eta_2) P_{\tau'}(\eta_1) \nonumber \\
			& \qquad +P_{\tau'}(\eta_1)P_{\tau''}(\eta_2, \eta_3) + P_{\tau''}(\eta_2, \eta_3) P_{\tau'}(\eta_1) + P_{\tau'}(\eta_2)P_{\tau''}(\eta_1, \eta_3) \nonumber \\
			& \qquad +P_{\tau''}(\eta_1, \eta_3) P_{\tau'}(\eta_2) + P_{\tau'}(\eta_3)P_{\tau''}(\eta_1, \eta_2) + P_{\tau''}(\eta_1, \eta_2) P_{\tau'}(\eta_3) \nonumber \\
			&\qquad +P_{\tau'''}(\eta_1, \eta_2, \eta_3) \big{)} \widetilde{N}(\iota). \nonumber
		\end{align}
		In particular, if $\eta_1 = \ldots = \eta_k = \eta$, the second and third directional derivatives become
		\begin{align}
			\label{eq:DNt2}
			D^2 \widetilde{N}(\iota; \eta^2) = \big{(}&2 P_{\tau'}(\eta)^2 + P_{\tau''}(\eta^2)\big{)} \widetilde{N}(\iota), \\[1mm]
			\label{eq:DNt3}
			D^3 \widetilde{N}(\iota; \eta^3) = \big{(}& 6 P_{\tau'}(\eta)^3 + 3 P_{\tau''}(\eta^2) P_{\tau'}(\eta) 
			+ 3 P_{\tau'}(\eta) P_{\tau''}(\eta^2) + P_{\tau'''}(\eta^3) \big{)} \widetilde{N}(\iota).
		\end{align}
                The corresponding derivatives of the parametrized forward map $\iota \mapsto \widetilde{\Lambda}(\iota)$ can be obtained from the above formulas by operating with $T$ from the left.
	\end{lemma}

        \begin{proof}
          The results follow by combining \eqref{eq:DkN} with the chain rule for Banach spaces.
        \end{proof}

        Explicit bounds on the operator norms of the derivatives $D^k \widetilde{N}(\iota) \in \mathscr{L}^k(\mathcal{B}, \mathscr{L}(\C^M_\diamond, \mathcal{H}))$ and $D^k \widetilde{\Lambda}(\iota) \in \mathscr{L}^k(\mathcal{B}, \mathscr{L}(\C^M_\diamond))$ could be straightforwardly deduced based on the general form for the differentiation formulas in Lemma~\ref{lemma:tilde_deriv}, assumed bounds on the derivatives of the parametrization $\iota \mapsto \tau(\iota)$ and \eqref{eq:Pbound}; cf.~\eqref{eq:DnNbound} and \eqref{eq:DnLambdabound}. However, we content ourselves with simply noting that for any $k \in \N$ and $\omega \in \mathcal{I}$, there exists a constant $C_k(\omega)>0$ and an open neighborhood $\mathcal{N}_\omega$ of $\omega$ in $\mathcal{I}$ such that
        \begin{equation}
          \label{eq:deri_bound_tilde}
       \big\| D^k \widetilde{\Lambda}(\iota) \big\|_{\mathscr{L}^k(\mathcal{B}, \mathscr{L}(\C^M_\diamond))} \leq C_k(\omega) \qquad \text{for all } \iota \in \mathcal{N}_\omega
       \end{equation}
        if the mapping $\iota \mapsto \tau(\iota)$ is $k$ times continuously differentiable; cf.~\eqref{eq:coer} and \eqref{eq:DnLambdabound}.

\section{Series reversion}
\label{sec:seriesreversion}
In this section we introduce the series reversion procedure for the SCEM with a general admittivity parametrization. The derivation closely follows the ideas in~\cite{garde2021series} with only minor modifications and extensions. In fact, with the trivial parametrization $\tau = {\rm id}$ and $\mathcal{I} = \mathcal{K}_+$, the deduced reversion formulas are exactly the same as in \cite{garde2021series}, although the ones presented in this work also implicitly account for the boundary admittivity. However, as the derivatives of the forward map corresponding to a general admittivity parametrization have terms that do not appear if $D^k\tau = 0$ for $k \in \N \setminus \{1\}$ (compare \eqref{eq:DNt1}, \eqref{eq:DNt2} and \eqref{eq:DNt3} to \cite[eq. (5.2)]{garde2021series}), we end up with more terms in the series reversion formulas as well. For this reason, we content ourselves with third order series reversion even though the same approach could be straightforwardly, yet tediously, extended to arbitrarily high orders. What is more, Remark~\ref{remark:nonsymmetric} comments on a case were the current-feeding and potential-measuring electrodes need not be the same, which is a case not explicitly covered in \cite{garde2021series}. Because a translation is an analytic mapping, we may assume without loss of generality that the origin of the Banach space $\mathcal{B}$ acts as the initial guess for the to-be-reconstructed parameters and, in particular, that the origin belongs to the open parameter set $\mathcal{I} \subset \mathcal{B}$; cf.~Section~\ref{sec:param}.

Let $\mathcal{W} \subset \mathcal{B}$ be a subspace that defines the admissible perturbation directions in our parametrization for the interior and boundary admittivities. Throughout this section, the target admittivity pair is defined as $(\sigma, \zeta) = \tau(\upsilon)$ for a fixed $\upsilon \in \mathcal{I} \cap \mathcal{W}$ and an analytic parametrization $\tau: \mathcal{I} \to \mathcal{K}_+$. Furthermore, define $F = D\widetilde{\Lambda}(0; \, \cdot \,)$ and $\mathcal{Y} = F(\mathcal{W})$. We work under the following, arguably rather restrictive, main assumption.

\begin{assumption}
	\label{ass:only}
	The Fr\'echet derivative $F: \mathcal{B} \to \mathcal{L}(\C^M_\diamond)$ is injective on $\mathcal{W}$.\footnote{Assumption~\ref{ass:only} can be satisfied by choosing the number of electrodes $M$ high enough compared to the dimension of a suitably constructed $\mathcal{W}$, if the contact admittances are fixed; see \cite{Lechleiter_08b, Harrach_2019} for more information. In general, since $\mathcal{W}$ must be finite-dimensional for the assumption to hold, one may try to numerically verify if $F$ is injective.}
\end{assumption}

Since $\mathcal{L}(\C^M_\diamond)$ is finite-dimensional,  the same must also apply to $\mathcal{W}$ by virtue of Assumption~\ref{ass:only}. In consequence, $F \in \mathcal{L}(\mathcal{W}, \mathcal{Y})$ has a bounded inverse $F^{-1} \in \mathcal{L}(\mathcal{Y}, \mathcal{W})$. Moreover, there obviously exists a bounded projection $Q \in \mathcal{L}(\mathcal{L}(\C^M_\diamond), \mathcal{Y})$ onto $\mathcal{Y}$ due to the finite-dimensionality of the involved subspaces.

Let us then get into the actual business. The third order Taylor expansion for $\widetilde{\Lambda}$ around the origin can be arranged into the form
\begin{equation*}
	F\upsilon = D\widetilde{\Lambda}(0; \upsilon) = \widetilde{\Lambda}(\upsilon) - \widetilde{\Lambda}(0) - \frac{1}{2}D^2 \widetilde{\Lambda}(0; \upsilon^2) - \frac{1}{6} D^3 \widetilde{\Lambda}(0; \upsilon^3) - \frac{1}{6}\int_0^1 (1-s)^3 D^4 \widetilde{\Lambda}(s \upsilon; \upsilon^4) \, {\rm d} s
\end{equation*}
assuming that $\upsilon \in \mathcal{W}$ is small enough so that the whole line segment $[0, \upsilon] = \{ \iota \in \mathcal{B} \ | \ \iota = t \upsilon, \ t \in [0,1] \}$ lies in $\mathcal{I}$. The remainder term is of order $O(\|\upsilon\|_{\mathcal{B}}^4)$ due to \eqref{eq:deri_bound_tilde}. Operating with $\mathcal{M} := F^{-1} Q \in \mathcal{L}(\mathcal{L}(\C^M_\diamond), \mathcal{W})$ from the left  yields
\begin{align}
	\label{eq:etaTaylor}
	\upsilon = F^{-1}Q F \upsilon = \mathcal{M}\big(\widetilde{\Lambda}(\upsilon) - \widetilde{\Lambda}(0)\big) - \frac{1}{2} \mathcal{M} D^2 \widetilde{\Lambda}(0; \upsilon^2) - \frac{1}{6} \mathcal{M} D^3 \widetilde{\Lambda}(0; \upsilon^3) + O(\|\upsilon\|_{\mathcal{B}}^4)
\end{align}
because by assumption $\upsilon \in \mathcal{W}$ and $Q$ is a projection onto $\mathcal{Y}$.

At this point, it is worthwhile to stop for a moment and consider what we have actually derived.
\begin{itemize}
	\item The residual term $\widetilde{\Lambda}(\upsilon) - \widetilde{\Lambda}(0)$ in \eqref{eq:etaTaylor} only contains values that are known. Namely, $\widetilde{\Lambda}(\upsilon)$ is the available data and $\widetilde{\Lambda}(0)$ corresponds to an initial guess $\upsilon = 0$.
	\item Since the second and third derivatives of $\widetilde{\Lambda}$ in  \eqref{eq:etaTaylor} are of orders $O(\|\upsilon\|_{\mathcal{B}}^2)$ and $O(\|\upsilon\|_{\mathcal{B}}^3)$, respectively, by virtue of \eqref{eq:deri_bound_tilde}, the expansion \eqref{eq:etaTaylor} also provides first and second order approximations for~$\upsilon$.
	\item Since $\mathcal{W}$ and $\mathcal{L}(\C^M_\diamond)$ are finite-dimensional, the operator $\mathcal{M}$ can simply be implemented as a pseudoinverse of $F: \mathcal{W} \to \mathcal{L}(\C^M_\diamond)$ after introducing inner products for the involved spaces.
\end{itemize}
Higher order approximations for $\upsilon$ can naturally be derived by applying the same technique to higher order Taylor expansions of $\widetilde{\Lambda}$ (cf.~Corollary~\ref{corollary:para}).

Our leading idea is to recursively insert \eqref{eq:etaTaylor} into its right-hand side, which results in an explicit formula that returns $\upsilon$ up to an error of order $O(\|\upsilon\|_{\mathcal{B}}^4)$. Let us start by making \eqref{eq:etaTaylor} more explicit by plugging in \eqref{eq:DNt2} and \eqref{eq:DNt3} at $\iota = 0$ composed with $T$:
\begin{align}
	\label{eq:Psinbrackets}
	\upsilon  &= \mathcal{M}\big(\widetilde{\Lambda}(\upsilon) - \widetilde{\Lambda}(0)\big) - \frac{1}{2} \mathcal{M} T \Big( 2 P_{\tau'}(0;\upsilon)^2 + P_{\tau''}(0;\upsilon^2) \Big) \widetilde{N}(0) \\ \nonumber
 & \qquad  - \frac{1}{6} \mathcal{M} T \Big( 6 P_{\tau'}(0;\upsilon)^3 + 3 P_{\tau''}(0;\upsilon^2) P_{\tau'}(0;\upsilon) \\
	\nonumber
	& \qquad \qquad \qquad + 3P_{\tau'}(0;\upsilon) P_{\tau''}(0;\upsilon^2) + P_{\tau'''}(0;\upsilon^3) \Big) \widetilde{N}(0) + O(\|\upsilon\|_{\mathcal{B}}^4).
\end{align}
    Here and in the following, $\mathcal{M}$ always operates on the entire composition of operators on its right. Take note that the bounded operators $\mathcal{M}$, $T$ and $\widetilde{N}(0)$ do not depend on $\upsilon$, and thus they do not affect the asymptotic accuracy of the representation. In the following, the dependence on the base point $\iota = 0$ is  suppressed for the sake of brevity.

As the first order estimate, we immediately obtain
\begin{equation}
	\label{eq:eta1}
	\eta_1 := \mathcal{M}\big(\widetilde{\Lambda}(\upsilon) - \widetilde{\Lambda}(0)\big) = \upsilon + O(\|\upsilon\|_{\mathcal{B}}^2),
\end{equation}
    which corresponds to a linearization of the original reconstruction problem. Replacing $\upsilon$ by $\eta_1$ in the second term on the right-hand side of \eqref{eq:Psinbrackets} and employing boundedness of the involved linear operators, it straightforwardly follows that
\begin{align}
	\label{eq:eta2}
	\eta_2 &:= - \frac{1}{2}  \mathcal{M} T \Big( 2 P_{\tau'}(\eta_1)^2 + P_{\tau''}(\eta_1^2) \Big) \widetilde{N}
	= - \frac{1}{2} \mathcal{M}  T \Big( 2 P_{\tau'}(\upsilon)^2 + P_{\tau''}(\upsilon^2) \Big) \widetilde{N} + O(\|\upsilon\|_{\mathcal{B}}^3).
\end{align}
Obviously, $\eta_2$ is altogether of order $O(\|\upsilon\|_{\mathcal{B}}^2)$. Moreover, making the substitution \eqref{eq:eta2} in \eqref{eq:Psinbrackets} leads to the conclusion $\upsilon = \eta_1 + \eta_2 + O(\|\upsilon\|_{\mathcal{B}}^3)$.

In order to derive the third order approximation, we start by reconsidering the second term on the right-hand side of \eqref{eq:Psinbrackets}. Replacing $\upsilon$ this time around with $\eta_1 + \eta_2 + O(\|\upsilon\|_{\mathcal{B}}^3)$ leads to
\begin{align}
       \label{eq:D2estimate}
	- \frac{1}{2}  \mathcal{M} T \Big( 2 P_{\tau'}(\upsilon)^2 + P_{\tau''}(\upsilon^2) \Big) \widetilde{N}
	&= - \frac{1}{2} \mathcal{M} T \Big( 2P_{\tau'}(\eta_1)^2 +  P_{\tau''}(\eta_1^2) \nonumber \\ \nonumber
	& \qquad \qquad \quad + 2P_{\tau'}(\eta_1)P_{\tau'}(\eta_2) + 2P_{\tau'}(\eta_2)P_{\tau'}(\eta_1)  \\ \nonumber
	&\qquad \qquad \quad + P_{\tau''}(\eta_1, \eta_2) + P_{\tau''}(\eta_2, \eta_1) \Big) \widetilde{N} + O(\|\upsilon\|_{\mathcal{B}}^4)\\[1mm]
	&= \eta_2 - \mathcal{M} T \Big( P_{\tau'}(\eta_1)P_{\tau'}(\eta_2) + P_{\tau'}(\eta_2)P_{\tau'}(\eta_1) \\ &\qquad \qquad \qquad + P_{\tau''}(\eta_1, \eta_2) \Big) \widetilde{N} + O(\|\upsilon\|_{\mathcal{B}}^4), \nonumber
\end{align}
    where all higher order terms have been collected in $O(\|\upsilon\|_{\mathcal{B}}^4)$. Let us then tackle the third term on the right-hand side of \eqref{eq:Psinbrackets}. As it is of third order in~$\upsilon$, inserting the first order estimate $\upsilon = \eta_1 + O(\|\upsilon\|_{\mathcal{B}}^2)$ is sufficient for our purposes:
\begin{align}
	\label{eq:D3estimate}
	-\frac{1}{6} \mathcal{M} T & \Big( 6 P_{\tau'}(\upsilon)^3 + 3 P_{\tau''}(\upsilon^2) P_{\tau'}(\upsilon)	+ 3P_{\tau'}(\upsilon) P_{\tau''}(\upsilon^2) + P_{\tau'''}(\upsilon^3) \Big) \widetilde{N} \nonumber \\
	=& -  \mathcal{M} T \Big( P_{\tau'}(\eta_1)^3 + \frac{1}{2} P_{\tau''}(\eta_1^2) P_{\tau'}(\eta_1)
	+ \frac{1}{2} P_{\tau'}(\eta_1) P_{\tau''}(\eta_1^2) + \frac{1}{6}P_{\tau'''}(\eta_1^3) \Big) \widetilde{N} + O(\|\upsilon\|_{\mathcal{B}}^4).
\end{align}
Combining the observations in \eqref{eq:D2estimate} and \eqref{eq:D3estimate}, it is natural to define
\begin{align}
	\label{eq:eta3}
	\eta_3 &= - \mathcal{M} T \Big( P_{\tau'}(\eta_1)^3 + \frac{1}{2} P_{\tau''}(\eta_1^2) P_{\tau'}(\eta_1) + \frac{1}{2} P_{\tau'}(\eta_1)  P_{\tau''}(\eta_1^2) + \frac{1}{6}P_{\tau'''}(\eta_1^3) \\ \nonumber
        	&  \qquad \qquad \ +  P_{\tau'}(\eta_1)P_{\tau'}(\eta_2) + P_{\tau'}(\eta_2)P_{\tau'}(\eta_1) + P_{\tau''}(\eta_1, \eta_2) \Big) \widetilde{N},
\end{align}
which is of order $O(\|\upsilon\|_{\mathcal{B}}^3)$.

Summarizing the information in \eqref{eq:Psinbrackets}--\eqref{eq:eta3}, we have demonstrated under Assumption~\ref{ass:only} that
\begin{equation}
	\label{eq:errors}
	\bigg\| \upsilon - \sum_{k=1}^K \eta_k \bigg\|_{\mathcal{B}} \leq C_K \|\upsilon \|_{\mathcal{B}}^{K+1}, \qquad K=1, \dots, 3,
\end{equation}
where $C_K>0$ is independent of small enough $\upsilon \in \mathcal{B}$. Take note that the definition of the components in the reconstruction is recursive: $\eta_1$ depends on the residual $\widetilde{\Lambda}(\upsilon) - \widetilde{\Lambda}(0)$, $\eta_2$ depends on $\eta_1$, and $\eta_3$ depends on both $\eta_1$ and $\eta_2$.  Although the Taylor series representation for $\widetilde{\Lambda}$ is available only if $\tau: \mathcal{I} \to \mathcal{K}_+$ is analytic, it is easy to check via Taylor's theorem that the parametrization actually only needs to be $K+1$ times continuously differentiable for the conclusion \eqref{eq:errors} to hold.

\begin{remark}
          \label{remark:nonsymmetric}
	  In the above derivation we tacitly assumed that the complete set of electrodes is available both for feeding currents and measuring voltages. However, this is not the case in,~e.g.,~geophysical electrical resistivity tomography~\cite{Ducut22}, where the injecting and measuring electrodes are separate. To overcome this limitation, let us define
          $$
          \C^M_{\rm in} = \{ J \in \C^M_\diamond \ | \ J_m = 0 \textrm{ if } E_m \text{ is not a current-feeding electrode} \}
          $$
          and
          $$
          \C^M_{\rm out} = \{ J \in \C^M_\diamond \ | \ J_m = 0 \textrm{ if } E_m \text{ is not a potential-measuring electrode} \}.
          $$
          The orthogonal projections of $\C^M_\diamond$ onto the subspaces $\C^M_{\rm in}$ and $\C^M_{\rm out}$ are denoted by $\mathcal{P}_{\rm in}$ and $\mathcal{P}_{\rm out}$, respectively. The above analysis remains valid as such if one redefines $F$ as
          $$
          F =  \mathcal{P}_{\rm out} D\widetilde{\Lambda}(0;\, \cdot \,) \mathcal{P}_{\rm in}
          $$
          and $\mathcal{M}: \mathcal{L}(\C^M_\diamond) \to \mathcal{W}$ via
          $$
          \mathcal{M}L = F^{-1}Q \mathcal{P}_{\rm out}L \mathcal{P}_{\rm in}.
          $$
With this modification, the projected relative data $\mathcal{P}_{\rm out}(\widetilde{\Lambda}(\upsilon) - \widetilde{\Lambda}(0)) \mathcal{P}_{\rm in}$ takes the role of the residual $\widetilde{\Lambda}(\upsilon) - \widetilde{\Lambda}(0)$. It is obvious that (a noisy version of) the former is available when only using the current-feeding and potential-measuring electrodes in their respective roles.
\end{remark}

\section{Numerical implementation}
\label{sec:implementation}
In this section we expand on the initial two-dimensional numerical setup in~\cite[Section~7]{garde2021series} by considering the more practical framework of the three-dimensional SCEM. In particular, we purposefully include some modeling errors in our setting so that their effect can be studied in the numerical experiments of Section~\ref{sec:numerics}, and we also employ so dense discretizations for the internal admittivity that stable inversion of the Fr\'echet derivative of the forward map is impossible without regularization. In the following, we only consider real-valued admittivities,~i.e.,~{\em conductivities}.

The observations about the computational complexity of the CM based numerical implementation in \cite[Section~7]{garde2021series} remain essentially the same in our setting even though the underlying sesquilinear form is different, the reconstruction of the contact conductivities is included in the algorithm and general parametrizations of the unknowns are considered. As mentioned in the previous section, allowing arbitrary parametrizations for the conductivities introduces a number of extra terms in the reversion formulas, but the recursive computational structure of \eqref{eq:eta1}, \eqref{eq:eta2} and \eqref{eq:eta3} is in any case very similar to that of the respective terms $F_1$, $F_2$ and $F_3$ in \cite{garde2021series}. In consequence, the increase in the computational complexity is still a mere constant multiplier independent of the discretizations when the order of the inversion scheme is increased, albeit the constant is now larger due to the derivatives of the composite operator appearing in the reversion formulas for a general parametrization. As with the standard parametrization of the conductivity in~\cite{garde2021series}, we still only have to compute directional derivatives instead of the complete tensor-valued higher order derivatives of the forward operator~$N$.

\subsection{Parametrization and discretization}
\label{sec:param}
In all our experiments, the target domain $\Omega$ is the perturbed unit ball depicted in~Figure~\ref{fig:recs}. To enable avoiding an inverse crime between the measurement simulation and reconstruction stages, the surface of $\Omega$ is triangulated at two resolutions, and the resulting high (H) and low (L) resolution surface triangulations, $T_{\rm H}$ and $T_{\rm L}$, are independently tetrahedralized into polygonal domains $\Omega_{\rm H}$ and $\Omega_{\rm L}$, respectively.
In particular, the electrode positions are not accounted for when forming the surface meshes, which causes a certain discrepancy in the electrode geometries between the two discretizations, as explained in more detail below. The denser FE mesh consists of $23\,630$ nodes and $95\,951$ tetrahedra, while the corresponding surface mesh has $15\,584$ nodes and $31\,164$ triangles. The sparser FE mesh has $4984$ nodes and $18\,597$ tetrahedra, with the corresponding surface mesh composed of $3730$ nodes and $7456$ triangles.

The tetrahedralizations $\Omega_{\rm H}$ and $\Omega_{\rm L}$ are clustered into $N_{\rm H} = 4000$ and $N_{\rm L}= 1000$ polygonal connected subsets of roundish shape and approximately the same size, respectively, in order to introduce the subdomains for piecewise constant representations of the conductivity. These subdomains are denoted $\omega_{*,i} \subset \Omega_*$, $i = 1, \ldots N_*$, where the subindex $*$ stands for H or L. We represent the domain log-conductivity as
$$
\kappa_* = \mu_\kappa + \sum_{i=1}^{N_*} \kappa_i \mathbf{1}_{*,i}, \qquad \kappa = (\kappa_1, \dots, \kappa_{N_*}) \in \R^{N_*},
$$
where $\mathbf{1}_{*,i}$ is the indicator function of $\omega_{*,i}$ and $\mu_\kappa \in \R$ corresponds to information on the expected log-conductivity level in $\Omega$, making $\kappa = 0$ a reasonable basis point for the reversion formulas of Section~\ref{sec:seriesreversion}. The high and low resolution parametrizations for the domain conductivity are then defined via
\begin{equation}
  \label{eq:sigma_param}
\sigma_*:
\left\{
\begin{array}{l}
\kappa \mapsto {\rm e}^{\kappa_*} = {\displaystyle {\rm e}^{\mu_\kappa} \sum_{i=1}^{N_*} {\rm e}^{\kappa_i} \mathbf{1}_{*,i}}, \\[4mm]
  \R^{N_*} \to L^\infty_+(\Omega_*).
\end{array}
\right.
\end{equation}
The parameter vector $\kappa \in \R^{N_*}$ is identified with a shifted piecewise constant log-conductivity via writing $\kappa = \sum_{i=1}^{N_*} \kappa_i \mathbf{1}_{*,i}$ with respect to the appropriate partitioning.

The numerical experiments are performed with $M=24$ electrodes. Their midpoints $x_m$, $m=1, \dots, M$, are distributed roughly uniformly over $\partial \Omega$; cf.~Figure~\ref{fig:recs}. Assuming the triangles of $T_*$ are closed sets, the electrodes for the two levels of discretization are defined as the open polygonal surface patches
$$
%% NH:
E_{*, m} = {\rm int} \Big( \bigcup \big\{t \in T_* \ | \ c_t \in B(x_m, 0.3)  \big\} \Big), \qquad m=1, \dots, M,
%E_{*, m} = {\rm int} \Big( \bigcup \big\{t \in T_* \ | \ c_t \cap B(x_m, 0.3) \not= \emptyset \big\} \Big), \qquad m=1, \dots, M,
$$
where $c_t$ is the centroid of the triangle $t$ and $B(x, r)$ denotes the open ball of radius $r$ centered at $x$. Because the surface triangulations $T_{\rm H}$ and $T_{\rm L}$ do not match, the corresponding electrodes do not exactly coincide either, i.e., $E_{\rm H,m} \not = E_{\rm L,m}$ for $m=1, \dots, M$. That is, we do not assume perfect prior knowledge on the electrode positions in the experiments where the measurement simulation and reconstruction are performed on different  meshes.

The contact conductivities are parametrized either as a function taking constant values on certain subsets of the electrodes and vanishing elsewhere or by a smooth function following the construction in~\cite{Candiani19}. The former is dubbed the {\em CEM parametrization} in reference to the standard approach of modeling the contact resistivity/conductivity as piecewise constant in the standard CEM, whereas the latter option is called the {\em smooth parametrization} since it is constructed in the spirit of the SCEM (cf.~\cite{Hyvonen17b}).

The CEM parametrization is only used for simulating measurement data, not for computing reconstructions. We do not assume there is contact everywhere on the surface patches identified as the electrodes,\footnote{In the terminology of \cite{darde2021electrodeless}, $E_{*,m}$ could be called an {\em extended electrode}.} but the corresponding contact regions are defined as
$$
%% NH:
e_{*, m} = {\rm int} \Big(\bigcup\big\{t \in T_* \ | \ c_t \in B(x_m, 0.2)  \big\}\Big) \subset E_{*, m}, \qquad m=1, \dots, M,
%e_{*, m} = {\rm int} \Big(\bigcup\big\{t \in T_* \ | \ c_t \cap B(x_m, 0.2) \not= \emptyset \big\}\Big) \subset E_{*, m}, \qquad m=1, \dots, M,
$$
i.e., by the same formula as the associated electrodes but with a smaller radius. The actual CEM parametrization is
\begin{equation}
  \label{eq:CEM_param}
\zeta_{*,{\rm CEM}}:
\left\{
\begin{array}{l}
{\displaystyle \theta \mapsto {\rm e}^{\mu_{\zeta}} \sum_{m=1}^{M} {\rm e}^{\theta_m} \frac{\tilde{\chi}_{{\rm *},m}}{|e_{{\rm *}, m}|}} , \\[4mm]
  \R^{M} \to L^\infty(\partial \Omega_*),
\end{array}
\right.
\end{equation}
where $\theta_m$ corresponds to the shifted log-conductance on $E_{*, m}$, $\mu_{\zeta} \in \R$ is the expected log-conductance, and $\tilde{\chi}_{*,m}$ is the indicator function of $e_{*, m}$.

With the smooth parametrization, employed for both measurement simulation and reconstruction on both discretizations, we aim to also demonstrate the functionality of the contact-adapting model of \cite{darde2021electrodeless} in a three-dimensional setup. That is, we do not only include the strength of the contact as a parameter in the smoothened model but also its relative location on the considered electrode. To this end, each electrode $E_{*,m}$ is projected onto an approximate tangent plane, obtained via a least squares fit, at the corresponding midpoint $x_m$. The coordinate system on the tangent plane is chosen so that the projection of the electrode just fits inside the square $[-1,1]^2$ in the local coordinates. There is an obvious freedom in choosing the orientation of the local coordinates, but we do not go into the details of our choice as we do not expect the precise orientation to have any meaningful effect on the numerical results. This construction defines a bijective change of coordinates
$$
\phi_{*,m}: E_{*,m} \to \phi_{*,m}(E_{*,m}) \subset [-1,1]^2, \qquad m = 1, \dots, M,
$$
on each electrode and for both levels of discretization, assuming the electrodes are small enough. On $[-1,1]^2$, we define a smooth surface conductivity shape
\begin{equation}
  \label{eq:smoothcontacts}
  \psi_{\xi}(y)=
  \left\{
  \begin{array}{ll}
    {\displaystyle \exp \! \left( a - \frac{a}{1 - \frac{|y - \xi |^2}{R^2} } \right)}, & \qquad   |y - \xi| < R, \\
    0 & \qquad \text{otherwise},
  \end{array}
  \right.
  \end{equation}
where the midpoint $\xi$ is considered an unknown in the reconstruction process, but the other two parameters are assigned fixed values $R = 0.6$ and $a = 4$. The former controls the width of the contact, whereas the latter fine-tunes its shape; in particular, the contact region on the parameter square is $D(\xi, R)$, i.e., the disk of radius $R$ centered at $\xi$.

The smooth conductivity parametrization on the $m$th electrode is defined as
\begin{equation}
  \label{eq:smth_component}
(\zeta_{*,{\rm smooth}})_m:
(\rho, \xi) \mapsto  {\rm e}^{\rho+\mu_{\zeta}} \frac{\psi_\xi \circ \phi_{*,m}}{\int_{E_{*,m}} \psi_\xi \circ \phi_{*,m} \, {\rm d} S } ,
\end{equation}
where ${\rm e}^{\rho+\mu_{\zeta}}$ is the net contact conductance on the electrode, i.e., the integral of the surface conductivity over the electrode, which makes $\rho+\mu_{\zeta}$ the log-conductance. All the strength and location parameter pairs $(\rho, \xi)$ of individual electrodes are collected into a single parameter vector~$\theta$, and the corresponding domain of definition $\mathcal{D}_{\rm smooth}$ for the smooth parametrization is defined to be the subset of $\R^{3M}$ defined by the condition that $\phi_{*,m}^{-1}(D(\xi, R)) \subset E_{*,m}$ for $m=1, \dots, M$, that is, the contact regions are required to be subsets of the respective electrodes. To summarize, we have arrived at the complete smooth parametrization:
\begin{equation}
  \label{eq:smooth_param}
\R^{3M} \supset \mathcal{D}_{\rm smooth} \ni \theta \mapsto \zeta_{*,{\rm smooth}}(\theta) \in L^\infty(\partial \Omega_*),
\end{equation}
where $\zeta_{*,{\rm smooth}}(\theta)$ behaves as indicated by \eqref{eq:smth_component} on $E_{*,m}$, with $(\rho, \xi)$ defined by the appropriate components of $\theta \in \mathcal{D}_{\rm smooth}$, and vanishes outside the electrodes.

The complete parametrization is finally obtained by combining \eqref{eq:sigma_param} with either \eqref{eq:CEM_param} or \eqref{eq:smooth_param} to form a mapping $\tau$ from the complete parameter vector $\iota = (\kappa, \theta) \in \mathcal{I} \subset \R^{N}$ to a pair of domain and boundary conductivities. Depending on the contact parametrization and the level of discretization, the number of degrees of freedom ranges from $N = N_{\rm L} + M = 1024$ to $N = N_{\rm H} + 3M = 4072$. It is straightforward to check that for any admissible $\iota$ in the sense of the above definitions, the resulting image $\tau(\iota)$ lies in a discretized version of $\mathcal{K}_+$. Moreover, the derivatives of the parametrization with respect to $\iota$ can be explicitly calculated in a straightforward but tedious manner. The various terms in the series reversion formulas can then be computed by combining these parameter derivatives  with the appropriate numerical solutions of the variational problems \eqref{eq:weak} and \eqref{eq:Pdef}, obtained by resorting to a custom FE solver built on top of scikit-fem~\cite{skfem2020}. We employ a mixed FE method with standard piecewise linear elements for the domain potential and certain macro-elements for the electrode potentials; see \cite{Vauhkonen97} for a comparable scheme.

\subsection{Regularization of the Fr\'echet derivative and Bayesian interpretation}
\label{sec:Moperator}

In the numerical computations, electrode current-to-voltage maps and their directional derivatives are represented as symmetric matrices with respect to a certain orthonormal basis of $\R^{M}_\diamond$; see,~e.g.,~\cite[Section~4]{Hyvonen18}. In particular, such a representation for (noiseless) data, say, $\Psi \in \R^{(M-1)\times (M-1)}$ is defined by only $M(M-1)/2 = 276$ real numbers when $M=24$. Taking into account the well-known severe illposedness of the inverse problem of EIT and the number of degrees of freedom in our parametrizations for the domain and boundary conductivities, it is clear that operating with $\mathcal{M}$ in the series reversion formulas must be performed in some regularized manner. In what follows, we denote by ${\rm vec}: \R^{(M-1)\times (M-1)} \to \R^{(M-1)^2}$ the columnwise vectorization operator.

Let $\Gamma_{\rm noise} \in \R^{(M-1)^2 \times (M-1)^2}$ and $\Gamma_{\rm pr} \in \R^{N \times N}$ be symmetric positive definite matrices. We define a regularized version for $\mathcal{M}$ appearing in the series reversion formulas via
\begin{equation}
  \label{eq:regM}
\mathcal{M}_{\rm R}:
\left\{
\begin{array}{l}
  \Psi \mapsto \eta_{\rm R}, \\[1mm]
  \R^{(M-1)\times (M-1)} \to \R^N,
\end{array}
\right.
\end{equation}
where $\eta_{\rm R}$ is the minimizer of a Tikhonov functional,
\begin{equation}
  \label{eq:tikhonov}
\eta_{\rm R} =
\arg\!\!\min_{\eta \in \R^N} \Big(\big\| {\rm vec} (D \widetilde{\Lambda}(0; \eta) - \Psi) \big\|_{\Gamma_{\rm noise}^{-1}}^2 + \| \eta \|_{\Gamma_{\rm pr}^{-1}}^2\Big).
\end{equation}
Here, the standard notation $\| z \|^2_{A} = z^{\rm T} A z$ is used. This choice of regularization is motivated by Bayesian inversion \cite{Kaipio04a}. To be more precise, $\eta_{\rm R}$ is the {\em maximum a posteriori} (MAP) or conditional mean estimate for $\eta$ in the equation $D \widetilde{\Lambda}(0; \eta) = \Psi$ if the vectorized data ${\rm vec}(\Psi)$ is contaminated by additive zero-mean Gaussian noise with the covariance $\Gamma_{\rm noise}$ and $\eta$ is {\em a priori} normally distributed with a vanishing mean and the covariance $\Gamma_{\rm pr}$.

When computing the first order estimate $\eta_1$, replacing $\mathcal{M}$ by $\mathcal{M}_{\rm R}$ in \eqref{eq:eta1} corresponds to solving the linearized reconstruction problem with the aforementioned assumptions on the additive measurement noise contaminating the data $\widetilde{\Lambda}(\upsilon)$ and prior information on the parameter vector $\upsilon$. However, due to the nonlinear and recursive nature of the formulas defining the higher order terms $\eta_2$ and $\eta_3$ in the reconstruction, there is no obvious reason why this same choice of regularization for $\mathcal{M}$ would also be well-advised in \eqref{eq:eta2} and \eqref{eq:eta3} from the Bayesian perspective. Be that as it may, we exclusively employ this same regularization for all occurrences of $\mathcal{M}$ in our numerical tests, although we acknowledge that a more careful study on the choice of regularization in connection to the series reversion technique would be in order.

Let us then introduce the prior models employed in \eqref{eq:tikhonov}. We resort to a standard Gaussian smoothness prior for the domain log-conductivity, with the covariance matrix defined elementwise as
	\begin{equation}
	\label{eq:domaincovariance}
		(\Gamma_\kappa)_{i,j} = \gamma^2_\kappa \exp\left( -\frac{|z_{*,i} - z_{*,j}|^2}{2 \lambda^2_\kappa} \right), \qquad i,j = 1,\dots,N_*,
	\end{equation}
where $z_{*,i}$ is the center of $\omega_{*,i}$ in the partitioning of $\Omega_*$. Moreover, $\gamma^2_\kappa$ and $\lambda_\kappa$ are the pointwise variance and the so-called correlation distance, respectively, for the shifted log-conductivity in \eqref{eq:sigma_param}. The shifted electrode log-conductance parameters, appearing in both the CEM parametrization,~i.e.~$\theta$ in \eqref{eq:CEM_param}, and the smooth parametrization,~i.e.~$\rho$ in \eqref{eq:smooth_param}, are assigned a diagonal covariance of the form $\Gamma_\rho = \gamma_\rho^2 I \in \R^{M \times M}$. The $M$ contact location parameters $\xi \in [-1,1]^2$, only included in the smooth parametrization \eqref{eq:smooth_param}, are also equipped with a diagonal covariance $\Gamma_\xi = \gamma_\xi^2 I \in \R^{2M \times 2M}$. For the CEM parametrization, the total prior covariance is $\Gamma_{\rm prior} = {\rm diag}(\Gamma_\kappa, \Gamma_\rho)$, whereas that for the smooth parametrization is $\Gamma_{\rm prior} = {\rm diag}(\Gamma_\kappa, \Gamma_\rho, \Gamma_\xi)$, assuming the log-conductance parameters appear before the location parameters in $\theta \in \R^{3M}$ of \eqref{eq:smooth_param}. According to this covariance structure, the domain log-conductivity values are {\em a priori} correlated, but no other prior correlation between different parameters is assumed. Observe that it is natural to assume that the parameter vector has a vanishing mean (cf.~\eqref{eq:tikhonov}) as information on the expected values for the domain log-conductivity and the electrode log-conductances can be included as the shifts $\mu_{\kappa}$ and $\mu_{\zeta}$ in \eqref{eq:sigma_param}, \eqref{eq:CEM_param} and \eqref{eq:smooth_param}.

	To describe the noise model, the assumed structure of the physical measurements must first be explained. Let $e_1, \dots, e_M$ denote the standard Cartesian basis vectors of $\R^M$. Although the available data is presented in \eqref{eq:tikhonov} with respect to an orthogonal basis $I^{(1)}, \dots, I^{(M-1)}$ of $\R^M_\diamond$, the original noisy measurements are simulated by using  $\hat{I}^{(m)} = e_m - e_{m+1}\in \R^M_\diamond$, $m=1, \dots, M-1$, as the physical current patterns. The resulting potential vectors in Cartesian coordinates are then contaminated with additive Gaussian noise with a diagonal covariance.
	To allow a more precise explanation, we first define two auxiliary matrices  $B = [I^{(1)} \dots I^{(M-1)}] \in \R^{M\times (M-1)}$ and  $\hat{B} = [\hat{I}^{(1)} \dots \hat{I}^{(M-1)}] \in \R^{M\times (M-1)}$. Take note that both $B$ and $\hat{B}$ are invertible as mappings from $\R^{M-1}$ to $\R^M_{\diamond}$, which means that their pseudoinverses $B^\dagger$ and $\hat{B}^\dagger$ satisfy $B B^\dagger|_{\R^M_{\diamond}} = \hat{B} \hat{B}^\dagger|_{\R^M_{\diamond}} = {\rm Id}$ and $B^\dagger B = \hat{B}^\dagger \hat{B} = {\rm Id}$. Furthermore, $B^\dagger$ and $\hat{B}^\dagger$ have the span of $\mathbf{1} := [1 \dots 1]^{\rm T} \in \R^M$ as their kernels.

The noiseless physical measurements corresponding to a parameter $\upsilon \in \mathcal{I}$ are defined as
$$
	U^{(m)}(\upsilon) = B \widetilde{\Lambda}(\upsilon) B^\dagger \hat{I}^{(m)}, \qquad m=1, \dots, M-1.
$$
The corresponding noisy measurements are
\begin{equation}
	\label{eq:noisy_meas}
    V^{(m)}(\upsilon) = U^{(m)}(\upsilon) + \hat{\vartheta}^{(m)} - \frac{1}{M} \sum_{j=1}^M \hat{\vartheta}^{(m)}_j \mathbf{1}, \qquad m=1, \dots, M-1,
\end{equation}
where  $\hat{\vartheta}^{(m)} \in \R^{M}$ is a Gaussian random variable with the expected value $0 \in \R^{M}$ and a diagonal covariance $\Gamma^{(m)}$ defined elementwise via
\begin{equation}
	\label{eq:noise_model}
    \Gamma^{(m)}_{i,i} = \big( \delta_1 \max_{\substack{j=1,\dots, M\\ m=1,\dots, M}}|U^{(m)}_j(0)| \big)^2 + \big(\delta_2 |U^{(m)}_i(0)|\big)^2, \qquad \delta_1, \delta_2 > 0.
\end{equation}
In other words, the components of the noise are independent, and each of them is further a sum of two independent Gaussian random variables whose expected values are zero and the standard deviations are essentially proportional to the largest measured electrode potential and the corresponding electrode potential, respectively. The last term on the right-hand side of \eqref{eq:noisy_meas} ensures that the measured potential vector has zero mean, which simply corresponds to a normalization of the data recorded by a measurement device.

By simultaneously considering all $M-1$ equations in \eqref{eq:noisy_meas} in a matrix form, one arrives at
\begin{equation}
	\label{eq:noisy_data}
    \Upsilon(\upsilon) :=  B^\dagger \mathcal{V}(\upsilon)\hat{B}^\dagger B = \widetilde{\Lambda}(\upsilon) + B^\dagger \hat{\Theta} \hat{B}^\dagger B,
\end{equation}
where $\mathcal{V}(\upsilon) = [V^{(1)}(\upsilon) \dots  V^{(M-1)}(\upsilon)] \in \R^{M \times (M-1)}$ contains the noisy physical measurements and $\hat{\Theta} = [\hat{\vartheta}^{(1)} \dots \hat{\vartheta}^{(M-1)}] \in \R^{M \times (M-1)}$ all components of the additive noise. Because the elements of $\Theta := B^\dagger \hat{\Theta} \hat{B}^\dagger B \in \R^{(M-1) \times (M-1)}$ are linear combinations of independent Gaussian random variables whose expected values are zero, they are themselves Gaussians with vanishing expected values but with a nontrivial covariance structure that defines $\Gamma_{\rm noise}$ employed in \eqref{eq:tikhonov}.

\begin{remark}
	\label{remark:G-N}
        In the following section, we compare the performance of the novel one-step reconstruction methods with sequential linearizations. To be more precise, the latter reconstructions are computed by the iteration
    $$
    	\hat{\upsilon}_{j+1} = \hat{\upsilon}_{j} + \mathcal{M}_{\rm R}\big(\hat{\upsilon}_j)\big(\Upsilon(\upsilon) - \widetilde{\Lambda}(\hat{\upsilon}_j) \big), \qquad j=0,1,2, \quad \hat{\upsilon}_0 = 0,
    $$
where $\Upsilon(\upsilon)$ is the available noisy data and $\mathcal{M}_{\rm R}(\hat{\upsilon}_j)$ is defined in the same way as the standard $\mathcal{M}_{\rm R}$ but with $D\widetilde{\Lambda}(0;\eta)$ replaced by $D\widetilde{\Lambda}(\hat{\upsilon}_j;\eta)$ in \eqref{eq:tikhonov}. This corresponds to a Levenberg--Marquardt type algorithm as the update $\hat{\eta}_{j+1} = \hat{\upsilon}_{j+1} - \hat{\upsilon}_j$ is the minimizer of the Tikhonov functional
        $$
        \big\| {\rm vec} \big(D \widetilde{\Lambda}(\hat{\upsilon}_j; \eta) - (\Upsilon(\upsilon) - \widetilde{\Lambda}(\hat{\upsilon}_j)) \big) \big\|_{\Gamma_{\rm noise}^{-1}}^2 + \| \eta \|_{\Gamma_{\rm pr}^{-1}}^2
          $$
          with respect to $\eta$; cf.~\eqref{eq:eta1}, \eqref{eq:eta2} and \eqref{eq:eta3}.
\end{remark}

\begin{remark}
	Our choice for the physical current patterns is motivated by the observation that injecting current patterns between two electrodes is a simpler task technologically than simultaneously maintaining nonzero currents on several electrodes. Although the current is always fed between electrodes with consecutive numbers, this does not actually imply physical proximity of these electrodes. In fact, our numbering for the electrodes does not essentially have any correlation with their positions on the object boundary.
\end{remark}

\section{Numerical experiments}
\label{sec:numerics}
In this section we present our numerical experiments. The first experiment statistically compares the performance of higher order series reversion with 1--3 sequential linearizations (cf.~Remark~\ref{remark:G-N}). In the second experiment, we test if the higher orders of convergence verified in the simplistic CM based setup of \cite[Section~7]{garde2021series} can be detected in our setting with a high-dimensional unknown, modeling errors and noise. To conclude, we show some representative reconstructions for a couple of target conductivities.

We evaluate the performance of the reconstruction methods via two indicators, both of which have absolute and relative versions.  For a given target parameter vector $\upsilon$, the absolute and relative residuals
\begin{equation}
  \label{eq:res}
  {\rm Res}(\upsilon_i, \upsilon) = \|\widetilde{\Lambda}(\upsilon_i) - \Upsilon(\upsilon)\|_{\rm F},\qquad
{\rm Res}_{\rm rel}(\upsilon_i, \upsilon) = \frac{\|\widetilde{\Lambda}(\upsilon_i) - \Upsilon(\upsilon)\|_{\rm F}}{\|\widetilde{\Lambda}(0) - \Upsilon(\upsilon) \|_{\rm F}}
\end{equation}
measure the convergence of the simulated measurements corresponding to the reconstructed parameters toward the measured data. Here, $\upsilon_i$ is the reconstructed parameter vector, defined for the series reversion scheme as
$$
\upsilon_i = \sum_{k=1}^i \eta_k, \qquad i=1, \dots, 3.
$$
When considering sequential linearizations, one replaces $\upsilon_i$ by $\hat{\upsilon}_i$ in \eqref{eq:res}; see Remark~\ref{remark:G-N}. The measurement $\Upsilon(\upsilon)$ is simulated as described in \eqref{eq:noisy_data}. The reference measurement $\widetilde{\Lambda}(0)$ corresponds to the initial guess, i.e., the expected values for the parameters. Take note that \eqref{eq:res} utilizes the Frobenius norm, not the norm weighted by the inverse noise covariance as in \eqref{eq:tikhonov}.

The second pair of indicators are the absolute and relative reconstruction errors in the domain $\log$-conductivity:
\begin{equation}
  \label{eq:error}
  {\rm Err} (\upsilon_i, \upsilon) = \|\kappa_i - \kappa\|_{L^2_*(\Omega)}, \qquad
{\rm Err}_{\rm rel} (\upsilon_i, \upsilon) = \frac{\|\kappa_i - \kappa\|_{L^2_*(\Omega)}}{\|\kappa\|_{L^2(\Omega_{\rm H})}},
\end{equation}
where $\kappa_i$ and $\kappa$ are the parts of $\upsilon_i$ and $\upsilon$, respectively, defining the shifted $\log$-conductivities of the reconstruction and the target. The definitions for sequential linearizations are analogous. If the discretizations of the log-conductivity are different in the data simulation and reconstruction steps, we resort to a nearest-neighbor projection between the corresponding meshes before computing the error in the $L^2(\Omega_{\rm H})$ norm. This explains the special notation $L^2_*(\Omega)$ for the norm appearing in \eqref{eq:error}.

We do not attempt to quantitatively compare the reconstructed contact conductivities with the true ones since it is nontrivial to define an indicator for measuring the difference between two electrode conductivities that follow different parametrizations (cf.~Remark~\ref{remark:CEM-smooth} at the end of this section). Moreover, when simulating data with the smooth contact conductivity parametrization, the location parameters $\xi$ in \eqref{eq:smooth_param} are chosen to be zero. This means that the contact region is always approximately at the middle of the respective electrode when data is simulated; the flexibility to reconstruct the locations of the smooth contacts is simply considered as a tool for compensating for the discrepancy between the CEM and smooth parametrizations if the data is simulated by the former and, as always, the reconstruction is formed with the latter.

The free parameters appearing in \eqref{eq:sigma_param}, \eqref{eq:CEM_param}, \eqref{eq:smooth_param}, \eqref{eq:tikhonov} and \eqref{eq:noise_model} are specified in Table~\ref{tab:experiment1setups} for six experimental cases C1--C6. The left-hand side of the table gives the parameter values, discretization level and type of contact parametrization for simulating the data; note that the covariance-defining parameters for the contact strengths and domain log-conductivity are needed for making random draws from the prior. The right-hand side gives the parameters employed when computing the reconstructions; note that the lower discretization level, i.e.~$\Omega_{\rm L}$, and the smooth contact parametrization are always used in the reconstruction step. The value of the background log-conductivity $\mu_{\kappa} = -3.0$ in Table~\ref{tab:experiment1setups} approximately corresponds to the conductivity of Finnish tap water (cf.,~e.g.,~\cite{Hyvonen17b}) if the true physical diameter of our domain of interest were two meters, whereas the expected value of the electrode log-conductances $\mu_{\zeta} = -3.0$ corresponds to a net resistance of about $20\,\Omega$ at each electrode.

\begin{table}[tb]
  \caption{Specifications of the six experimental cases. The lower level of discretization $\Omega_{\rm L}$ and the smooth contact model are used for reconstruction. The parameter values $\gamma_\xi = 0.02$, $\mu_{\kappa} = -3.0$ and $\mu_{\zeta}=-3.0$ are common for all experiments.}
  \begin{center}
		\begin{tabular}{l|llrrr|rrr}
			& \multicolumn{5}{c|}{Measurement} & \multicolumn{3}{c}{Reconstruction} \\
			& $\Omega$ & ${\rm contact}$ & $(\delta_1, \delta_2)$ \% & $\gamma_\rho$ & $(\gamma_\kappa, \lambda_\kappa)$ & $(\delta_1, \delta_2)$ \% & $\gamma_\rho$ & $(\gamma_\kappa, \lambda_\kappa)$ \\ \hline
C1 & $\Omega_{\rm L}$ & Smooth & (0.005, 0.05) & 0.1 & (0.1, 1.0) & (0.01, 0.1) & 0.1 & (0.1, 1.0) \\
C2 & $\Omega_{\rm L}$ & Smooth & (0.01, 0.1) & 0.3 & (0.6, 0.6) & (0.05, 0.5) & 0.3 & (0.5, 0.7) \\
C3 & $\Omega_{\rm H}$ & CEM & (0.005, 0.05) & 0.1 & (0.1, 1.0) & (0.01, 0.1) & 0.07 & (0.1, 1.0) \\
C4 & $\Omega_{\rm H}$ & CEM & (0.005, 0.05) & 0.3 & (0.6, 0.6) & (0.05, 0.5) & 0.2 & (0.6, 0.6) \\
C5 & $\Omega_{\rm H}$ & CEM & (0.005, 0.05) & 0.3 & (0.6, 0.6) & (0.05, 0.5) & 0.2 & (0.5, 0.7) \\
C6 & $\Omega_{\rm H}$ & CEM & (0.01, 0.1) & 0.3 & (0.6, 0.6) & (0.05, 0.5) & 0.2 & (0.5, 0.7)
		\end{tabular}
                \end{center}
		\label{tab:experiment1setups}
\end{table}

The cases C1--C2 correspond to an inverse crime because the simulation of data and reconstruction are performed using exactly the same discretization. In the other four cases C3--C6, the measurement data is simulated using the denser discretization and the CEM model for the electrode contacts, which introduces a considerable discrepancy in modeling compared to the coarser discretization and the smooth contact model used in the reconstruction step. The noise level and the informativeness of the prior also vary between the six cases; our expectation is that the higher order methods function more stably when the prior is more informative,~i.e.,~the to-be-reconstructed parameters are expected to lie closer to the initial guess $\upsilon = 0$. In order to achieve reasonable stability for the statistical experiments, the prior used for reconstruction is often slightly more restrictive than the one used for random draws in the simulation step,  and the reconstruction methods work under an assumption of a considerably higher noise level than is actually employed in the simulation of data. That is, the regularization in \eqref{eq:tikhonov} is chosen to be stronger than the available prior information would suggest, which is expected to compensate for the numerical and modeling errors not accounted for in the additive noise model in \eqref{eq:noise_model}.

\subsection{Experiment~1}

In the first experiment, the following steps are carried out for all six cases listed in Table~\ref{tab:experiment1setups}: $1000$ samples of target parameters are drawn from the prior and the corresponding measurements are simulated according to the specifications on the left-hand side of the table. Subsequently, five different reconstructions are computed for each simulated noisy measurement according to the specifications on the right-hand side of the table. The considered reconstruction methods are the regularized first (1), second (2) and third (3) order series reversions (see~\eqref{eq:eta1}, \eqref{eq:eta2}, \eqref{eq:eta3} and \eqref{eq:tikhonov}) as well as the two- (1, \!1) and three-step (1, \!1, \!1) sequential linearizations (see~Remark~\ref{remark:G-N}).\footnote{The first order series reversion and one-step sequential linearization are in fact the same method.} In order to exclude potentially diverging cases from the statistical analysis, the 200 random draws of target parameters $\upsilon$ that resulted in the largest absolute residuals $\|\widetilde{\Lambda}(0) - \Upsilon(\upsilon)\|_{\rm F}$ are excluded before computing the mean values for the indicators ${\rm Res}_{\rm rel}$ and ${\rm Err}$ over the constructed sample of reconstructions. The resulting numbers for the six cases C1--C6 and the considered reconstruction methods are listed in Table~\ref{tab:experiment1measures}. Moreover, Figure~\ref{fig:exp1} visualizes the distributions of ${\rm Res}_{\rm rel}$ and ${\rm Err}_{\rm rel}$ over the full set of 1000 samples for the cases C1 and C5.

\begin{table}
  \caption{
    Sample means of the performance indicators ${\rm Res}_{\rm rel}$ and ${\rm Err}$ over those $800$, of a total of $1000$, random draws from the prior that resulted in the smallest absolute residuals.
    The rows correspond to the different tests specified in Table~\ref{tab:experiment1setups}. The columns correspond to different reconstruction methods: (1), (2) and (3) are the first, second and third order one-step series reversions, whereas (1, 1) and (1, 1, 1) stand for two and three steps of sequential linearizations (cf.~Remark~\ref{remark:G-N}). The trivial method of resorting to the initial guess $\upsilon_0 = 0$ is labeled (0) and $\mathbb{E}_{\upsilon}$ denotes the sample mean operation with respect to the drawn targets $\upsilon$.}
        \begin{center}
		\begin{tabular}{l|rrrrr|rrrrrr}
		  & \multicolumn{5}{c|}{$\log_{10} \!\big(\mathbb{E}_{\upsilon}({\rm Res}_{\rm rel} (\, \cdot \,, \upsilon))\big)$}
                   & \multicolumn{6}{c}{$\log_{10}\big(\mathbb{E}_{\upsilon}({\rm Err} (\, \cdot, \, , \upsilon))\big)$} \\
			& (1) & (2) & (3) & (1, \! 1) & (1, \!1, \!1) & \begin{tabular}{r} \end{tabular} (0) & (1) & (2) & (3) & (1, \!1) &  (1, \!1, \!1) \\
			\hline
C1 & -1.12 & -1.86 & -1.86 & -2.10 & -2.11 & -0.25 & -1.20 & -1.28 & -1.29 & -1.29 & -1.27 \\
C2 & -0.25 & -0.70 & -0.54 & -0.97 & -1.99 & 0.57 & 0.11 & -0.09 & -0.14 & -0.19 & -0.26 \\
C3 & -0.73 & -1.43 & -1.81 & -1.99 & -2.41 & -0.25 & -0.60 & -0.52 & -0.50 & -0.62 & -0.70 \\
C4 & -0.47 & -0.90 & -1.06 & -1.31 & -2.29 & 0.58 & 0.23 & 0.19 & 0.16 & 0.12 & 0.11 \\
C5 & -0.44 & -0.90 & -1.05 & -1.27 & -2.23 & 0.58 & 0.19 & 0.14 & 0.11 & 0.07 & 0.05 \\
C6 & -0.44 & -0.90 & -1.05 & -1.27 & -2.16 & 0.58 & 0.19 & 0.14 & 0.11 & 0.07 & 0.05 \\
		\end{tabular}
                \end{center}
		\label{tab:experiment1measures}
	\end{table}

Let us first consider the results for the easier of the two `inverse crime cases' C1, where the noise level is low and the priors for the domain log-conductivity and the electrode log-conductances are very informative. According to Table~\ref{tab:experiment1measures}, both ${\rm Res}_{\rm rel}$ and ${\rm Err}$ decrease on average when the order of the series reversion method is increased, although the decrease is no longer substantial between the second and third order methods. The two- and three-step linearization schemes result in approximately as low domain error in the log-conductivity as the second and third order series reversions, but in considerably smaller values for the relative residual. There is not much performance difference between two and three sequential linearizations, suggesting that two linearization steps are essentially enough for convergence. These conclusions are confirmed by the left-hand images of Figure~\ref{fig:exp1}. Take note that case C1 is the only one where the higher order series reversions are able to equally compete with (two) sequential linearizations in light of either of the two performance indicators considered in Table~\ref{tab:experiment1measures}.

The other `inverse crime case' C2 considers a higher noise level and considerably less informative priors for the domain log-conductivity and the electrode log-conductances. This time, the mean values of ${\rm Res}_{\rm rel}$ are much larger for the higher order series reversion methods than for two or three steps of sequential linearizations, with the third order series reversion scheme actually increasing the residual on average compared to its second order counterpart. However, the difference in the performance of the series reversion and sequential linearizations is not as pronounced when measured by the quality of the log-conductivity reconstructions.

\begin{figure}[t]
	\centering
	\subfloat{\includegraphics[width=\textwidth]{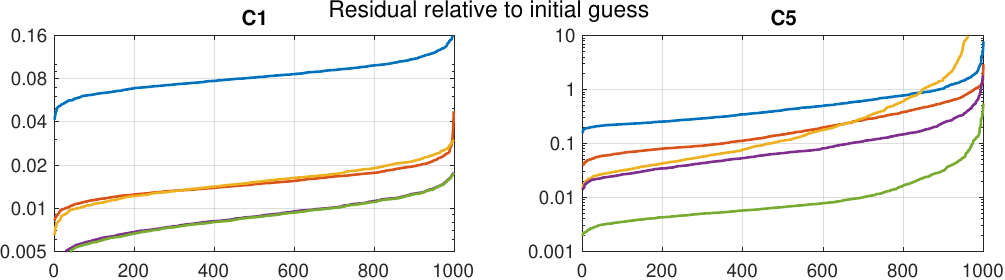}}\\
	\subfloat{\includegraphics[width=\textwidth]{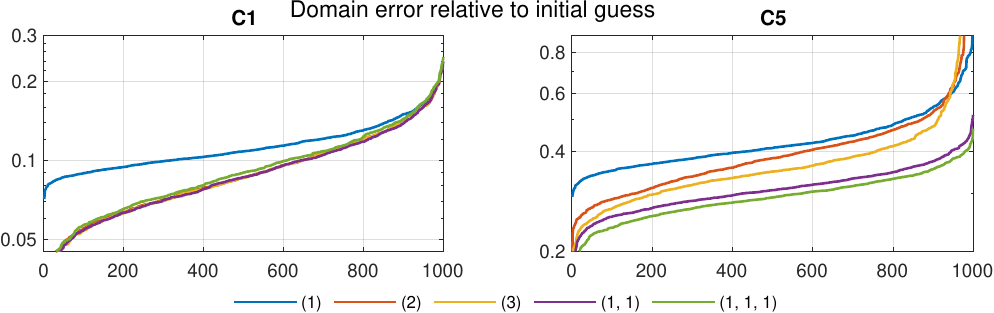}}
	\caption{Distributions of the relative residual of the forward problem ${\rm Res}_{\rm rel}$ and the relative domain conductivity error of the reconstructions ${\rm Err_{\rm rel}}$ over a sample of 1000 target parameter vectors drawn from the prior distributions in cases C1 and C5. The reconstructions were computed for regularized first (1), second (2) and third (3) order series reversions, and for two- (1, \!1) and three-step (1, \!1, \!1) sequential linearizations, with reconstruction and regularization parameters as defined for cases C1 and C5. For each method, the values of ${\rm Res}_{\rm rel}$ and ${\rm Err_{\rm rel}}$ are separately sorted in ascending order on the horizontal axis.}
	\label{fig:exp1}
\end{figure}

Case C3 alters C1 to another direction: the prior and noise model remain the same as for C1, but the denser discretization and the CEM parametrization for the contact conductivities are employed when simulating measurement data. This leads to considerable modeling errors. The behavior of ${\rm Res}_{\rm rel}$ is comparable to case C1 for both the series reversion and sequential linearizations, but the log-conductivity reconstructions are on average much worse, with the performance of the series reversion actually decreasing in this regard as a function of its order. It is worth mentioning that even the observed decrease in relative residuals for the series reversion methods is achieved only by including the contact locations $\xi$ of the smooth parametrization \eqref{eq:smooth_param} as unknowns in the inversion, that is, if $\xi = 0$ were fixed, ${\rm Res}_{\rm rel}$ would also increase on average.

Like C3, the remaining cases C4--C6 also consider settings where the data simulation is performed with the denser discretization and the CEM parametrization for the electrode contacts, while the reconstructions are computed using the sparser discretization and the smooth contact model. All three cases consider the less informative prior model for generation of data, and their differences are related to the noise level and whether a more conservative prior is assumed for the reconstruction step compared to how the random draws are performed. The conclusions for all four cases C4--C6 are essentially the same,~i.e.,~the differences in parameter values between these cases do not seem to have any significant effect on the results.  According to both indicators of Table~\ref{tab:experiment1measures}, the performance of the series reversion improves as a function of its order, although the difference between the second and third order is significantly smaller than that between the first and second order when the relative residual is considered. Moreover, two sequential linearizations clearly outperform the higher order series reversion methods with respect to both ${\rm Res}_{\rm rel}$ and ${\rm Err}$. For case C5, these conclusions are seconded by the right-hand images of~Figure~\ref{fig:exp1}.

Although in all tests the higher order series reversion methods are clearly inferior to sequential linearizations, in none of the cases considered in Table~\ref{tab:experiment1measures} do the higher order series reversion schemes completely break down. Moreover, one could argue that their relative performance compared to the sequential linearizations remains approximately the same over all six cases C1--C6, with all methods suffering roughly equally from the modeling errors in cases C3--C6. This is maybe a bit surprising since EIT is known to be sensitive to modeling errors \cite{Barber88,Breckon88,Kolehmainen97}, and it would  be intuitive to expect the higher order series reversions to further highlight this sensitivity due to their recursive nature. However, this observed robustness is probably partially due to the exclusion of the 200 `most difficult' parameter samples from the mean values listed in Table~\ref{tab:experiment1measures}. Indeed, the higher order series reversion methods have a tendency to occasionally `explode' as illustrated by the right ends of some yellow and red curves in Figure 5.1, indicating that in some cases a considerable proportion of the reconstructions fail to improve compared to the initial guess. This phenomenon can be considered to be analogous to the behavior of higher order polynomial interpolation or extrapolation for noisy data. In particular, this might suggest a decreased radius of convergence in the higher order methods compared to the first order method, but in this study we did not push this question further.

\begin{remark}
  \label{remark:CEM-smooth}
  According to our experiments, the (shifted) log-conductance parameters in \eqref{eq:CEM_param} and \eqref{eq:smooth_param} do not seem to have the same effect on the electrode measurements. That is, even though plugging in the same value for $\mu_\zeta + \theta_m$ in \eqref{eq:CEM_param} and $\mu_\zeta + \rho$ in \eqref{eq:smooth_param} gives the same net integral of the surface conductivity over the considered electrode, it seems that the actual effect of the contact on the potential at the electrode is considerably different. This phenomenon emphasizes the modeling errors in cases C3--C6 as it makes the initial guess $\upsilon = 0$ nonoptimal when data is simulated by resorting to the CEM parametrization but reconstructions are formed with the smooth parametrization. See~\cite[Section~4.3]{Hyvonen17b} for related observations.
\end{remark}

\begin{remark}
	\label{remark:asym_compl}
	Because the asymptotic computational complexity of any of our series reversion methods is the same as that of a single regularized linearization, it would be natural to also compare sequential linearizations with sequential series reversions defined in the obvious manner: one first computes a reconstruction by series reversion, say, $\upsilon_i$, then utilizes it as the basis point for the next series reversion, and so on. The reason for not considering,~e.g.,~a method labeled as $(2, 2)$ in our numerical tests is that for all considered test cases C1--C6, sequential series reversions perform statistically approximately as well as the corresponding number of sequential linearizations. That is, the performance of,~e.g.,~the method $(2, 2)$ is about as good as that of $(1, 1)$.
	%TK:
	In fairness, it should also mentioned that with our Python implementation, the considered FE discretizations, and the number of electrodes and unknown parameters, it is not yet obvious that the computational cost of two series reversions of different orders are of the same asymptotic computational complexity. One would probably need to consider even denser FE discretizations to observe this phenomenon. In other words, building and Cholesky-factorizing the system matrix corresponding to the sesquilinear form \eqref{eq:sesqui} plus forming and inverting the Fr\'echet derivative do not clearly dominate the other steps required by a series reversion scheme in our tests (cf.~\cite{garde2021series}).
\end{remark}

\subsection{Experiment~2}
The second experiment mimics the simplistic convergence test of \cite[Section~7.2]{garde2021series} in our three-dimensional setting with a high-dimensional unknown and modeling errors. The following computations are performed for cases C1 and C5. A single target parameter vector, say, $\hat{\upsilon}$ is drawn from the prior distribution defined on the left-hand side of the appropriate row in Table~\ref{tab:experiment1setups}, and the actual target is then formed via scaling as $\upsilon = s \hat{\upsilon}$ for some $s \in [0.2,10]$. The reconstructions by the considered five methods are computed according to the specifications of the considered case in Table~\ref{tab:experiment1setups}, but with the standard deviations of the domain log-conductivity and the contact log-conductance, $\gamma_\kappa$ and $\gamma_\rho$,  scaled by the utilized $s$.\footnote{Recall that the noise level depends on the size of the absolute data (cf.~\eqref{eq:noise_model}), so it does not go to zero as a function of $s$.} This procedure is repeated for multiple values of $s \in [0.2,10]$, and the absolute indicators ${\rm Res}$ and ${\rm Err}$ are computed for each of them. Note that the considered scaling range covers multiple orders of magnitude in terms of the domain conductivity and the contact conductances due to the employed logarithmic parametrizations.

Figure~\ref{fig:exp2} illustrates the performance indicators ${\rm Res}$ and ${\rm Err}$ as functions of the scaling parameter $s$ for both cases C1 and C5. In all subimages of Figure~\ref{fig:exp2}, the range of $s$ is restricted to a subinterval of $[0.2,10]$ to exclude uninteresting settings where the reconstruction methods essentially fail. Comparing the results to \cite[Example~7.2]{garde2021series}, we immediately observe a significant (relative) performance degradation with the higher order series reversions. The worse performance probably results from our experiment setup being significantly more complicated than in the earlier article.

\begin{figure}[t]
\centering
	\subfloat{\includegraphics[width=\textwidth]{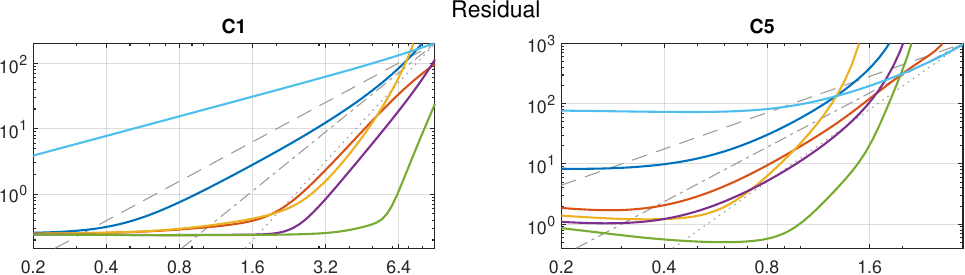}}\\
	\subfloat{\includegraphics[width=\textwidth]{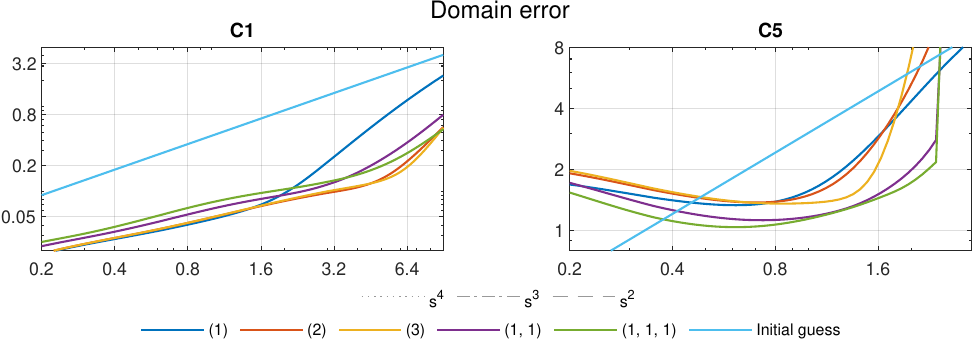}}
	\caption{Absolute residual ${\rm Res}$ (top) and domain error ${\rm Err}$ (bottom) as functions of a scaling factor $s>0$ for a single random draw from the prior in cases C1 (left) and C5 (right). Both the random draw and the corresponding (pointwise) standard deviations $\gamma_\kappa$ and $\gamma_\rho$ are scaled by the considered $s$ when computing the reconstructions, but the scaling factor does not affect the noise level that depends on the absolute measurements as indicated by \eqref{eq:noise_model}.}
	\label{fig:exp2}
\end{figure}

Let us first digest the `inverse crime case' C1. The top left image in Figure~\ref{fig:exp2} indicates that the first order series reversion, or a single regularized linearization, demonstrates approximately quadratic convergence in ${\rm Res}$ up to a scaling factor of the order $s=0.5$, below which the additive noise, which is proportional to the absolute measurement data (cf.~\eqref{eq:noise_model}), presumably starts to dominate and prohibits further convergence. The other reconstruction methods exhibit higher orders of convergence in ${\rm Res}$, but understandably they are also unable to converge beneath the noise level that is essentially independent of $s$. On the other hand, the bottom left image in Figure~\ref{fig:exp2} shows that the higher order series reversion methods and two- and three-step linearizations clearly demonstrate initial higher orders of convergence in the domain log-conductivity reconstruction error ${\rm Err}$, but all of them eventually settle with the same low order of convergence as a single regularized linearization

The right-hand images in Figure~\ref{fig:exp2} demonstrate the effect of the modeling errors induced by simulating the measurement data by the CEM parametrization for the contact conductivity but computing the reconstructions by resorting to the smooth contact model in case C5. Let us first consider the visualization of the absolute residuals ${\rm Res}$ on the top right in Figure~\ref{fig:exp2}. Due to the discrepancy between the two contact models considered in Remark~\ref{remark:CEM-smooth}, even the initial guess $\upsilon = 0$ does not produce an accurate estimate for the measured data when $s$, and thus also the target $\upsilon = s \hat{\upsilon}$, converges to zero. The higher order series reversions as well as the two- and, especially, three-step linearizations do a much better job in reducing the residual when $s$ decreases, but a single linearization is not flexible enough to overcome the modeling error between the CEM and smooth models for the contact conductivity. All considered reconstruction methods dominate the initial guess $\upsilon = 0$ when measured in the absolute domain log-conductivity error ${\rm Err}$ up to the level of about $s=0.5$, below which the modeling errors seem to prohibit further convergence. Naturally, the initial guess is not affected by such modeling errors, but it steadily converges toward the target $s \hat{\upsilon}$ as $s$ goes to zero. It is difficult to draw any reliable conclusions on the relative performance of the considered five reconstruction methods based on the right-hand images of Figure~\ref{fig:exp2}; the only thing that seems rather certain is that all of them become unreliable for scaling factors that are considerably larger than $s=1$ in case C5.

\subsection{Example reconstructions}

\begin{figure}[t!] % These both use seed=0 electrodes, noise and domain (only in (a)).
	\begin{center}
		\subfloat[\label{fig:c5rec}]{\includegraphics[width=\textwidth]{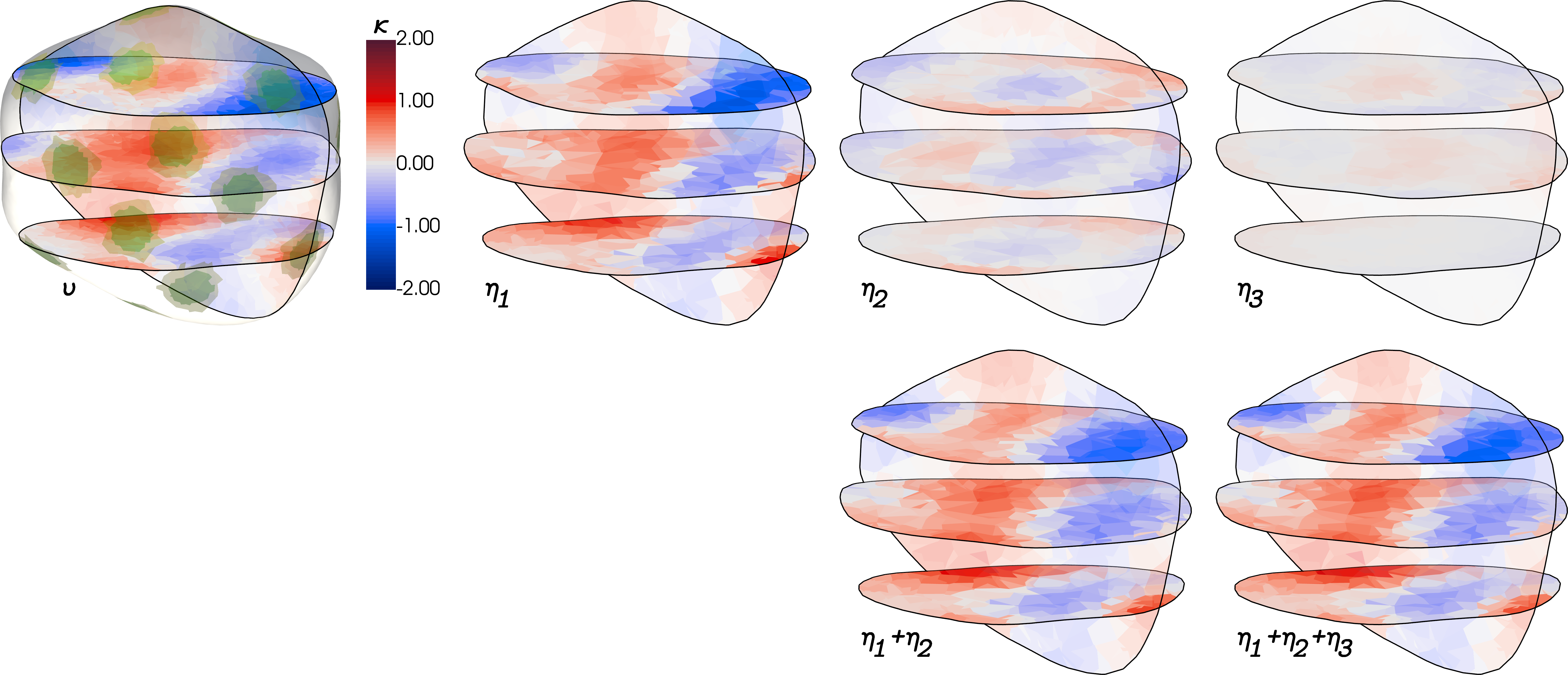}}\\
		\subfloat[\label{fig:oodrec}]{\includegraphics[width=\textwidth]{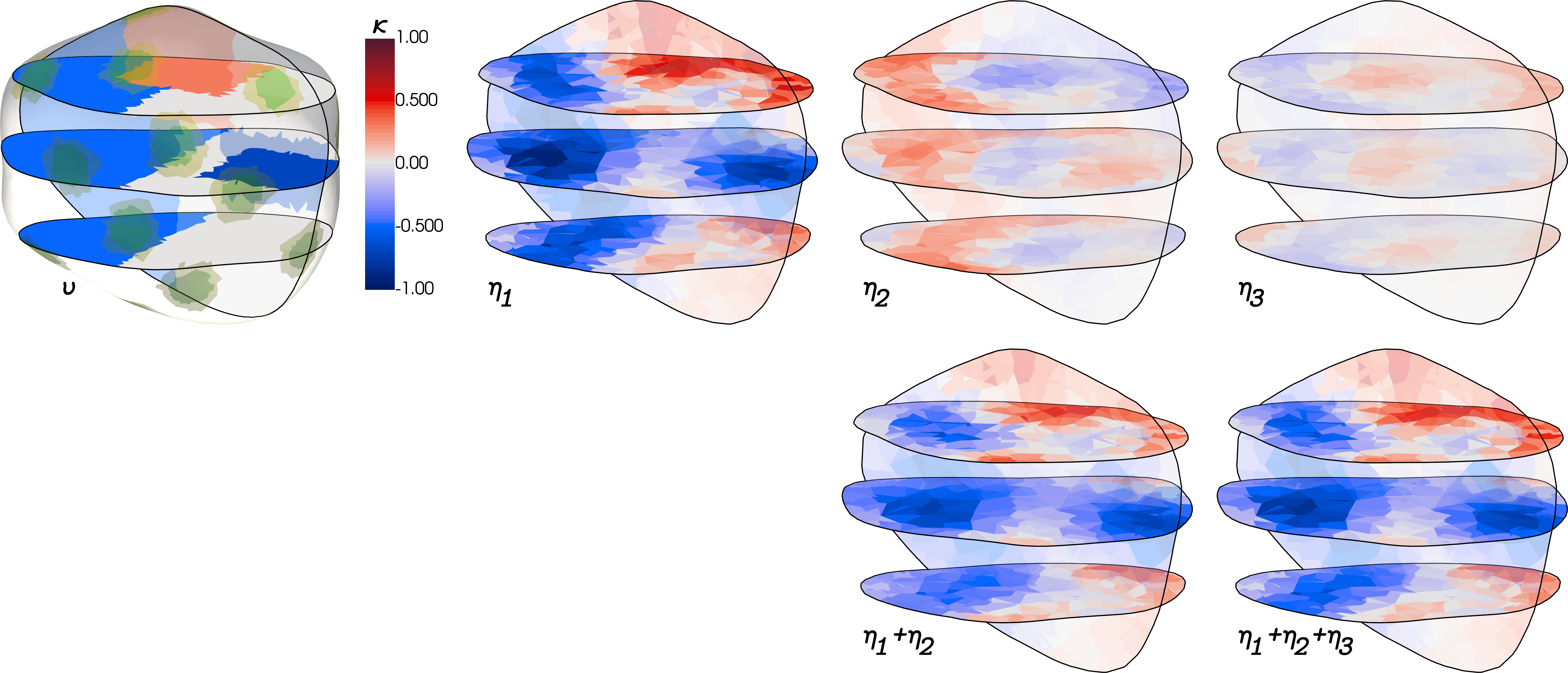}}
	\end{center}
	\caption{Two target log-conductivities and corresponding reconstructions by the series reversion methods. One vertical and three horizontal slices of the (shifted) target log-conductivities are depicted in the panels labeled by $\upsilon$. The same panels also show the forward-facing electrodes as translucent yellow surface patches and the corresponding reconstructed contact regions as green patches. The other panels visualize the first, second and third order series reversion reconstructions of the log-conductivities as well as their individual components in the order indicated by the labeling. Top: An example target log-conductivity and the corresponding reconstructions in case C5 of Table~\ref{tab:experiment1setups}. Bottom: A piecewise constant target log-conductivity with four partially overlapping, roughly rectangular, subdomains. The other parameters for the data simulation and reconstruction steps were chosen according to case C5.}
	\label{fig:recs}
\end{figure}

%% NH:
\begin{figure}[t!] % These both use seed=0 electrodes, noise and domain (only in (a)).
	\begin{center}
		\subfloat[\label{fig:c5rec2D}]{\includegraphics[width=\textwidth]{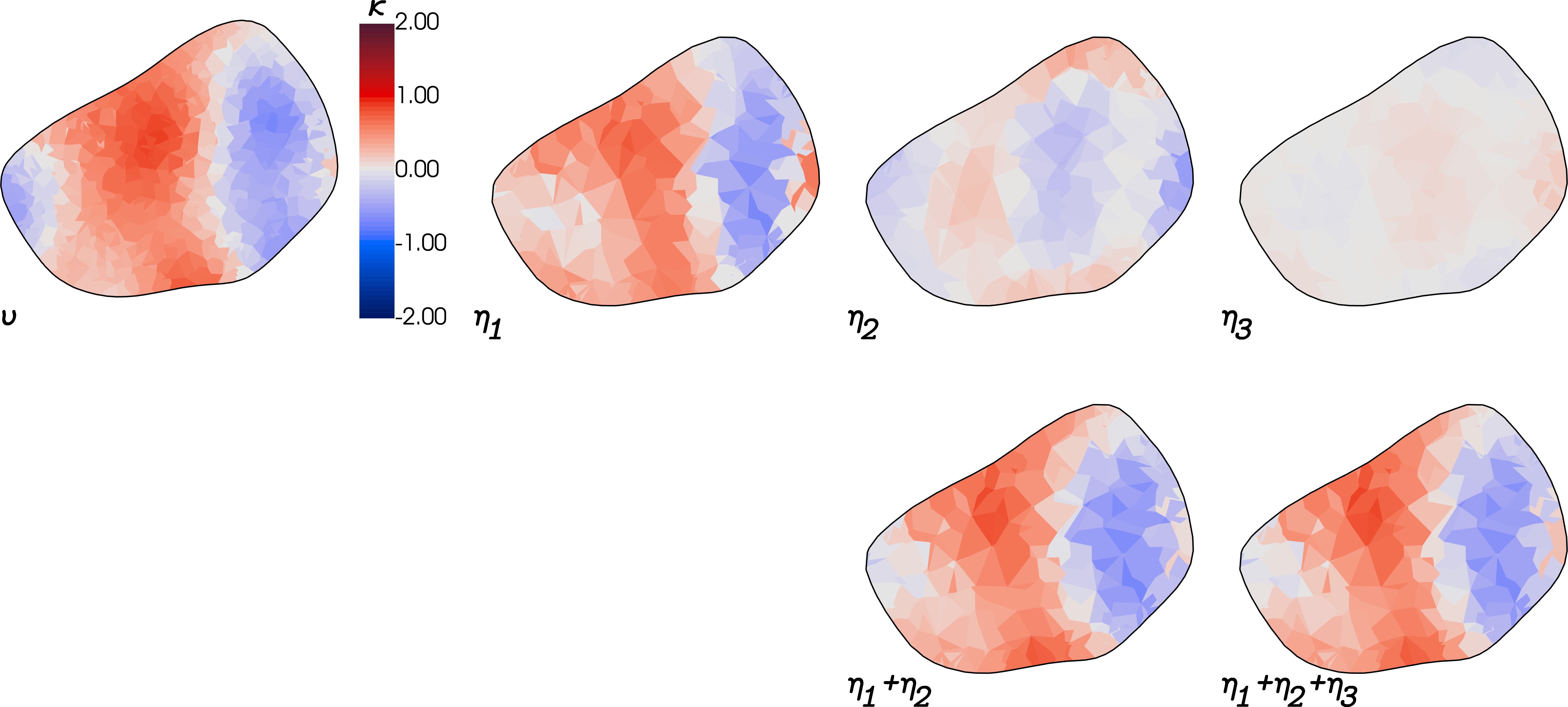}}\\
		\subfloat[\label{fig:oodrec2D}]{\includegraphics[width=\textwidth]{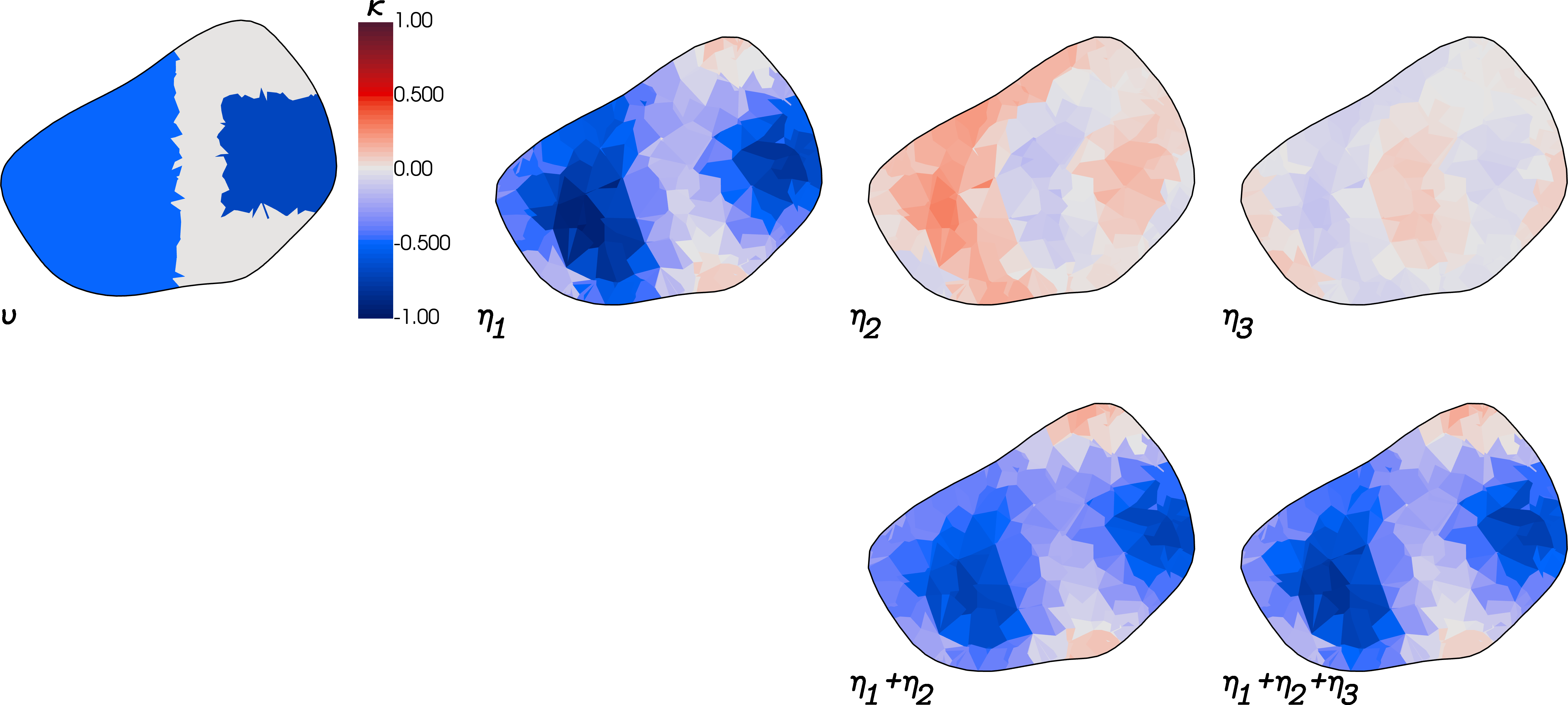}}
	\end{center}
	\caption{The central horizontal cross-sections in the twelve panels of Figure~\ref{fig:recs} presented as two-dimensional images. The ordering of the images is the same as in Figure~\ref{fig:recs}.}
	\label{fig:recs2D}
\end{figure}

\begin{table}
  \caption{The values of the performance indicators ${\rm Res}_{\rm rel}$ and ${\rm Err}$ for the series reversion reconstructions in Figure~\ref{fig:recs}. The numbering of the methods is as in Table~\ref{tab:experiment1setups}}
  \begin{center}
  \begin{tabular}{l|rrr|rrrr}
				& \multicolumn{3}{c|}{$\log_{10}\!\big({\rm Res}_{\rm rel}(\upsilon_i, \upsilon)\big)$} & \multicolumn{4}{c}{$\log_{10}\!\big({\rm Err}(\upsilon_i, \upsilon)\big)$} \\
				& (1) & (2) & (3) & \begin{tabular}{r} \end{tabular} (0) & (1) & (2) & (3) \\
				\hline
				(a) & -0.42 & -1.10 & -1.13 & 0.53 & 0.14 & 0.13 & 0.10 \\
				(b) & -0.43 & -0.83 & -1.17 & 0.42 & 0.20 & 0.20 & 0.19
  \end{tabular}
  \end{center}
  \label{table:viimeinen}
  \end{table}

Figure~\ref{fig:recs} shows two example target log-conductivities as well as associated reconstructions by the series reversion methods. The reconstructions in Subfigure~\ref{fig:c5rec} correspond to case C5, that is, the target parameter vector $\upsilon$ is drawn and the reconstructions formed according to the specifications of case C5 in Table~\ref{tab:experiment1setups}. The reconstructions in Subfigure~\ref{fig:oodrec} are also computed according to the specifications of case C5, but they correspond to a piecewise constant log-conductivity target consisting of three regions where the shifted log-conductivity in \eqref{eq:sigma_param} differs from the expected value of $\kappa = 0$. The other parameters for this piecewise constant example are chosen as in case C5.
%% NH:
The central horizontal cross-sections of all panels in Figure~\ref{fig:recs} are shown as two-dimensional images in Figure~\ref{fig:recs2D}.

%% NH:
	Based on Figures~\ref{fig:recs} and \ref{fig:recs2D}, an immediate, and perhaps disappointing, observation is that it is difficult to detect a considerable difference between the first order $\upsilon_1 = \eta_1$ and the higher order, i.e.~$\upsilon_2 = \eta_1 + \eta_2$ and $\upsilon_3 = \eta_1 + \eta_2 + \eta_3$, reconstructions. Although the values of the standard performance indicators ${\rm Res}_{\rm rel}$ and ${\rm Err}$ listed in Table~\ref{table:viimeinen} do in fact demonstrate a slight improvement for both examples and error indicators as functions of the reversion order, the differences in the quality of the domain log-conductivity reconstructions are simply too small to be visually significant.

	Comparing the first, second and third order components of the reconstructions, i.e.~$\eta_1$, $\eta_2$ and $\eta_3$, a damped oscillatory behavior is observed in their values. This is especially evident between $\eta_2$ and $\eta_3$, for which the positive and negative areas seem to be essentially swapped. According to our experience, this same oscillatory phenomenon is often encountered with most target domains, not just the ones depicted in~Figure~\ref{fig:recs}.

\section{Concluding remarks}
\label{sec:conclusions}

The goal of this paper was to introduce the series reversion methods of~\cite{garde2021series} as reconstruction schemes for three-dimensional electrode-based EIT under (considerable) modeling and numerical errors that are unavoidable in many real-life settings. It was demonstrated that the series reversion methods can indeed be robustly implemented in a practical setting, and their performance is arguably better than that of a single regularized linearization. However, a reconstruction method based on sequential linearizations, namely a Levenberg--Marquardt scheme, performed on average better than the higher order one-step series reversion methods in our statistical tests. Applying series reversions in a sequential manner resulted in approximately as good performance as the same number of sequential linearizations; see Remark~\ref{remark:asym_compl}.

According to our theoretical results as well as the numerical tests in \cite{garde2021series}, the series reversion methods demonstrate higher order convergence rates if the discretization of the unknowns is sparse enough to allow the Fr\'echet derivative of the forward map to be injective (and stably invertible). This condition was not met in our numerical studies, but the Fr\'echet derivative was inverted by resorting to certain kind of Tikhonov regularization that was motivated by Bayesian inversion. An interesting topic for future studies is to consider how regularization should be applied to the Fr\'echet derivative of the forward map in order to retain (some) good qualities of the series reversion methods in settings where the injectivity  of the Fr\'echet derivative cannot be guaranteed by considering a low enough number of unknown parameters.

\section*{Acknowledgments}
	We are grateful for Tom Gustafsson for guidance in the use of scikit-fem. The computations for statistical experiments were performed using computer resources within the Aalto University School of Science ``Science-IT'' project.

\bibliographystyle{siam}
\bibliography{series-refs}

\end{document}